\documentclass[11pt]{article}

\usepackage{amsmath}
\usepackage{amssymb}
\usepackage{amsthm}
\usepackage{graphicx}
\usepackage{color}
\usepackage{bbm}
\usepackage{cancel}
\usepackage[normalem]{ulem}

\usepackage[colorlinks,linkcolor=red,anchorcolor=blue,citecolor=green]{hyperref}

\DeclareSymbolFont{calletters}{OMS}{cmsy}{m}{n}
\DeclareSymbolFontAlphabet{\mathcal}{calletters}
\def \D{\mathbb{D}}

\def \R{\mathbb{R}}

\def \N{\mathbb{N}}

\newcommand{\dbN}{\mathbb{N}}
\newcommand{\dbP}{\mathbb{P}}

\DeclareMathOperator{\Expect}{\mathbb E}

\DeclareMathOperator{\trace}{tr}

\numberwithin{equation}{section}

\renewcommand{\leq}{\leqslant}
\renewcommand{\le}{\leqslant}
\renewcommand{\geq}{\geqslant}

\newcommand{\cC}{\mathcal{C}}

\newcommand{\cF}{\mathcal{F}}

\newcommand{\cL}{\mathcal{L}}
\newcommand{\cM}{\mathcal{M}}

\newcommand{\cP}{\mathcal{P}}
\newcommand{\cT}{\mathcal{T}}
\newcommand{\cW}{\mathcal{W}}
\newcommand{\cX}{\mathcal{X}}

\renewcommand{\d}{\ensuremath{\mathrm{d}}}
\newcommand{\transpose}{\mathsf{T}}
\newcommand{\ie}{{\it{i.e.}}}
\newcommand{\abs}[1]{\left\vert#1\right\vert}
\newcommand{\norm}[1]{\left\|#1\right\|}

\newcommand{\ba}{\begin{array}}
\newcommand{\ea}{\end{array}}
\newcommand{\bea}{\begin{eqnarray}}
\newcommand{\eea}{\end{eqnarray}}
\newcommand{\beaa}{\begin{eqnarray*}}
\newcommand{\eeaa}{\end{eqnarray*}}

\def\limsup{\mathop{\overline{\rm lim}}}

\def\1{{\bf 1}}

\def\:{\!:\!}

\newtheorem{thm}{Theorem}[section]
\newtheorem{lem}[thm]{Lemma}
\newtheorem{cor}[thm]{Corollary}
\newtheorem{prop}[thm]{Proposition}
\newtheorem{rem}[thm]{Remark}
\newtheorem{eg}[thm]{Example}
\newtheorem{defn}[thm]{Definition}
\newtheorem{assu}[thm]{Assumption}

\def\1{{\bf 1}}

\def\ie{\textit{i.e.}}

\title{Quantitative weak propagation of chaos \\for McKean--Vlasov branching diffusion processes}

\author{
	Wenjing Cao\thanks{Department of Mathematics, The Chinese University of Hong Kong. wjcao25@link.cuhk.edu.hk}
	\and Zhenjie Ren\thanks{LaMME, Universit\'e Evry Paris-Saclay. zhenjie.ren@univ-evry.fr}
        \and Xiaolu Tan\thanks{Department of Mathematics, The Chinese University of Hong Kong. xiaolu.tan@cuhk.edu.hk, Research supported by Hong Kong RGC General Research Fund (projects 14302921).}
}
\date{\today}

\begin{document}

\maketitle

\abstract{
	We study in this paper the weak propagation of chaos for McKean--Vlasov diffusions with branching,
	whose induced marginal measures are nonnegative finite measures but not necessary probability measures.
	The flow of marginal measures satisfies a non-linear Fokker--Planck equation, along which we provide a functional It\^o's formula.
	We then consider a functional of the terminal marginal measure of the branching process,
	whose conditional value is solution to a Kolmogorov backward master equation.
	By using It\^o's formula and based on the estimates of second-order linear and intrinsic functional derivatives of the value function, 
	we finally derive a quantitative weak convergence rate for the empirical measures of the branching diffusion processes with finite population.
}

\vspace{1mm}


\vspace{1mm}


\section{Introduction}\label{intro}

A classical McKean--Vlasov dynamic is a diffusion process governed by some stochastic differential equation (SDE) with coefficients dependent on marginal distribution of the solution itself, namely,
\[\d X_t=b\big(t,X_t,\text{Law}(X_t)\big)\,\d t+\sigma\big(t,X_t,\text{Law}(X_t)\big)\,\d W_t,\]
where $W=(W_t)$ is a Brownian motion.
Such dynamics are motivated by the study of large-population systems of mean-field interacting particles, when the population size tends to infinity,
and the limit property is the so-called propagation of chaos, cf.~\cite{kac1956foundations,mckean1966class,mckean1967propagation,sznitman1991topics}, etc.
Recently, McKean--Vlasov SDEs have become an increasingly important topic due to their broad applications across diverse fields, 
such as statistical physics (cf.~\cite{benedetto1997kinetic,carrillo2010particle}), mean-field games (cf.~\cite{carmona2018probabilistic,cis/1183728987,lasry2007mean}), economics and finance (cf.~\cite{garnier2013large,lacker2019mean}), and optimization (cf.~\cite{conforti2023game,du2023self,kazeykina2024ergodicity,nitanda2022convex}, etc.).

\vspace{0.5em}

In contrast to the classical McKean--Vlasov (or mean-field) setting, where the population size remains fixed, we study in this paper McKean--Vlasov diffusions with branching, in which the total population size evolves over time.
The study of branching diffusion processes dates back to the seminal works \cite{ikeda1968branching} and \cite{skorokhod1964branching}, which investigated in particular its link to the semilinear partial differential equations (PDEs) as extension of Feynman--Kac formula.
This connection has since been extended to broader classes of semilinear PDEs, motivating the development of novel Monte Carlo methods; see, e.g., \cite{henry2019branching,henry2014numerical}.
More recently, the branching diffusion processes with mean-field interaction has also been studied in the literature. 
In \cite{fontbona2015non}, the authors analyzed the convergence of a birth--death diffusion system with convolution-type interactions and proved well-posedness of the limiting Fokker–Planck equation as the number of particles tends to infinity. 
A quantitative weak propagation of chaos was subsequently obtained in~\cite{fontbona2022quantitative} by using coupling techniques.
In \cite{claisse2024mckean}, the authors studied the McKean--Vlasov branching SDE in a general path-dependent setting, establishing well-posedness and a propagation of chaos result (without rate) by using the weak convergence technique.
We also mention the recent work~\cite{wang2025distribution}, which studies distribution-dependent birth–death processes on discrete state spaces, and the paper \cite{claisse2023mean} which is on the mean-field game with branching diffusion process.

\vspace{0.5em}

Our work builds upon the framework introduced in~\cite{claisse2024mckean}, with the aim of establishing a quantitative rate of convergence for the weak propagation of chaos in branching diffusions with general interaction terms.
More precisely, let $\mu_T$ denote the marginal measure induced by the McKean--Vlasov branching diffusion at final time $T$ and $\mu^N_T$ denote the empirical measure induced by the finite-population branching process, we aim to establish that, for sufficiently smooth functional $G$,
\[\abs{\Expect[G(\mu_T^N)]-G(\mu_T)}=O(N^{-1}), ~\mbox{as}~N \longrightarrow \infty.\]
In the classical (non-branching) McKean--Vlasov setting, such weak propagation of chaos problems have been studied through the analysis of backward Kolmogorov equations and the use of functional derivatives on the space of probability measures.
The idea dates back to \cite{10.1214/19-EJP298}, where the authors used a ``master equation'' to approximate the Nash equilibrium for mean-field games.
This strategy was further developed and applied to various weak propagation of chaos studies,
including McKean--Vlasov SDEs with general coefficients (cf.~\cite{chassagneux2022weak}) and uniform in time convergence (cf.~\cite{delarue2025uniform}). 
In a recent paper \cite{chen2025improved}, improved weak convergence rates have been obtained for mean-field Langevin dynamics, via similar methods. 
Following the classical approach, we define the value function as the conditional value of the terminal function along the flow of marginal measure induced by the branching diffusions.
We then show that the value function satisfies a Kolmogorov backward equation (see \eqref{eq:bKol}) on the space of nonnegative finite measures.
Based on the estimates on the (second order) linear and intrinsic derivatives of the value function, together with our It\^o's formula in Theorem \ref{ito_formula} along the flow of the branching marginal measures, we finally deduce a convergence rate for the finite-population branching diffusion processes.
To obtain the estimates on the derivatives of the value function, a key argument is to relate the branching dynamics on $\R^d$ to a non-branching dynamic on $\R^{d+1}$ (see \eqref{lifted_MeanField}) by using the uniqueness results of the nonlinear Fokker--Planck equation in \cite{claisse2025optimal}.
This allows us to apply the analysis as in~\cite{10.1214/15-AOP1076} to classical McKean--Vlasov SDEs.

\vspace{0.5em}

The remainder of this article is organized as follows.
%Notations of this work are stated in the rest of Section~\ref{intro}.
In Section~\ref{results}, we introduce some notations, the main assumptions as well as our main theorems.
The proofs are completed in Section~\ref{proofs}.
Some intermediate results that are crucial to main theorems, along with their proofs, are provided in Appendices \ref{append_ito}, \ref{append_flow}, \ref{lifted_wellposedness_proof}, and \ref{proof_intrinsic_is_linear}.

\section{Main results}\label{results}

We first provide some notations and definitions in Section~\ref{preliminary},
and then introduce the McKean--Vlasov branching diffusion as well as the finite-population particle systems in Section \ref{subsec:MKV_branch_diffus}.
Main assumptions and theorems are then provided in Section~\ref{main_theorems}.

\subsection{Preliminaries}\label{preliminary}

Given a metric space $(\mathcal{X},d_\cX)$, we denote by $\cP(\cX)$ the space of all Borel probability measures on $\cX$, and by
\[\cP_p(\cX):=\Big\{\mu\in\cP(\cX):\,\int_\cX d_\cX(x,x_0)^p\,\mu(\d x)<\infty\quad\text{for some }x_0\in\cX\Big\}\]
the space of probability measures with finite $p$th moment for $p\geq 1$.
For $\mu,\nu\in\cP_p(\cX)$, the $p$-Wasserstein distance between them is denoted by
\[\cW_p(\mu,\nu):=\inf_{\xi\sim\mu,\eta\sim\nu}\Expect[d_\cX (\xi,\eta)^p]^{1/p}.\]
Analogously, we denote by $\cM(\cX)$ the space of nonnegative finite measures, and by $\cM_p(\cX)$ that of nonnegative finite measures with finite $p$th moment for $p\geq 1$.
For each Borel set $B\subset\cX$, $\cM(B)\subset\cM(\cX)$ denotes the space of nonnegative finite measures supported on $B$, while $\cM_p(B)$, $\cP(B)$ and $\cP_p(B)$ can be similarly defined.
For all $\mu\in\cM(\cX),f\in\mathbb{L}^1(\mu)$, we write
\[\langle f,\mu\rangle:=\int_\cX f(x)\,\mu(\d x).\]
Similarly, we write, for a signed measure $\nu=\nu^+-\nu^-$ where $\nu^+,\nu^-\in\cM(\R^d)$,
\[\langle f,\nu\rangle:=\langle f,\nu^+\rangle-\langle f,\nu^-\rangle,\]
for all bounded measurable functions $f$.

\vspace{0.5em}

As in \cite{claisse2024mckean}, we use the Ulam--Harris--Neveu notation to describe the labels of the particles in the branching process, \ie~let
\[k=k_1k_2...k_m\in\mathbb{K}:=\{\varnothing\}\cup\bigcup_{n=1}^\infty \dbN^n.\]
For each pair of labels $k=k_1...k_m,k^\prime=k^\prime_1...k^\prime_{m^\prime}\in\mathbb{K}$,
we denote by
\[kk^\prime:=k_1...k_m k^\prime_1...k^\prime_{m^\prime}\in\mathbb{K}\]
the concatenation of them, and define the space
\[E:=\Big\{e=\sum_{k\in K}\delta_{(k,x^k)}: K\subset\mathbb{K}\text{ finite},\,x^k\in\R^d,k\nprec k^\prime\text{ for all }k,k^\prime\in K \Big\},\]
where $\delta_\cdot$ denotes the Dirac measure.
The metric on $E$ is defined by
\[d_E(e,e^\prime):=\sum_{k\in K\cap K^\prime}\Big(\abs{x^k-y^k}\wedge 1\Big)+\#(K\triangle K^\prime)\]
for all $e=\sum_{k\in K}\delta_{(k,x^k)},\,e^\prime=\sum_{k\in K^\prime}\delta_{(k,y^k)}\in E$,
where $K\triangle K^\prime:=(K\setminus K^\prime)\cup (K^\prime\setminus K)$, and $\#\cdot$ denotes the number of elements in a finite set.
For $\phi=(\phi^k)_{k\in\mathbb{K}}:\mathbb{K}\times\R^d\to\R$ and $e=\sum_{k\in K}\delta_{(k,x^k)}\in E$, we write
\[\langle\phi,e\rangle:=\sum_{k\in K}\phi^k(x^k).\]
If $\phi^k=f:\R^d\to\R$ for each $k\in\mathbb{K}$, we write
\[\langle f,e\rangle:=\langle\phi,e\rangle=\sum_{k\in K}f(x^k).\]
In particular, 
\[\langle 1,e\rangle=\# K,\]
and we observe that
\[\cP_1(E)=\Big\{m\in \cP(E):\,\int_E\langle1,e\rangle\,m(\d e)<\infty \Big\}.\]
We also introduce the bounded Lipschitz distance, previously used in \cite{fontbona2022quantitative,piccoli2016properties}, as a metric on $\cM(\R^d)$.
\begin{defn}[Bounded Lipschitz distance]\label{wasserstein}
    For finite measures $\mu,\nu\in\cM(\R^d)$, the bounded Lipschitz distance is defined by 
    \[\mathbf{d}(\mu,\nu):=\sup_{\norm{f}_{\infty},\, \norm{f}_{\text{Lip}}\leq1}\abs{\langle f,\mu-\nu\rangle}.\]
    Here $\norm{\cdot}_{\text{Lip}}$ denotes the Lipschitz constant, with respect to the standard Euclidean norm $\abs{\cdot}$.
\end{defn}
\begin{rem}
	$\mathrm{(i)}$ By \cite[Theorem 8.3.2]{bogachev2007measure}, the space $(\cM(\R^d),\mathbf{d})$ is a Polish space, whose topology coincides with that of weak convergence.

	\vspace{0.5em}
	
	\noindent $\mathrm{(ii)}$ In \cite{claisse2025optimal,claisse2023mean}, the authors used a so-called extended Wasserstein distance $\bar{\cW}_p\,(p\geq1)$ on $\cM(\R^d)$, which is in fact equivalent to the bounded Lipschitz distance in the sense of \eqref{eq:equiv_bddLip}.
	
	The extended Wasserstein distance $\bar{\cW}_p$ is introduced as follows.
	For a metric space $(\cX,d_\cX)$, one introduces a cemetery point $\partial\notin \cX$ and let $\bar{\cX}:=\cX\cup\partial$ be the extended space.
By fixing some $x_0\in\cX$ and defining
\[d_{\bar{\cX}}(x,\partial):=d_{\cX}(x,x_0)+1,\,d_{\bar{\cX}}(x,\cdot)|_{\cX}:=d_\cX(x,\cdot),\quad x\in\cX,\]
$(\bar{\cX},d_{\bar{\cX}})$ is still a metric space.
For $\mu,\nu\in\cM_p(\bar{\cX})$ with the same mass $m=\mu(\bar{\cX})=\nu(\bar{\cX})\geq 0$, we write
\[\cW_{p,m}:=\inf_{\lambda\in\Lambda(\mu,\nu)}\Big(\int_{\bar{\cX}\times\bar{\cX}}\big(d_{\bar{\cX}}(x,y)\big)^p\,\lambda(\d x,\d y)\Big)^{\frac{1}{p}},\]
where $\Lambda(\mu,\nu)$ denotes the set of all finite (non-negative) Borel measures on $\bar{\cX}\times\bar{\cX}$ with marginals $\mu,\nu$.
Consequently, the extended Wasserstein distance between $\mu,\nu\in\cM_p(\cX)$ is defined by
\[\bar{\cW}_p(\mu,\nu):=\cW_{p,m}(\bar{\mu}_m,\bar{\nu}_m),\]
where $m\geq \mu(\cX)\vee\nu(\cX)$ and
\[\bar{\mu}_m(\cdot):=\mu(\cdot\cap\cX)+(m-\mu(\cX))\delta_\partial,\quad\bar{\nu}_m(\cdot):=\nu(\cdot\cap\cX)+(m-\nu(\cX))\delta_\partial.\]
For $p=1$, we observe by Kantorovich duality that
\begin{equation}\label{extended_kantorovich}
    \bar{\cW}_1(\mu,\nu)=\sup_{\norm{\bar{f}}_{\text{Lip}(\bar{\cX})}\leq 1}\abs{\int_{\cX}\big(\bar{f}(x)-\bar{f}(\partial)\big)\,(\mu-\nu)(\d x)}.
\end{equation}
Then for $\cX=\R^d,\,d_\cX(\cdot,\cdot)=\abs{\cdot-\cdot}\wedge1$, we can directly verify by \eqref{extended_kantorovich} that, for all $\mu,\nu\in\cM(\R^d)=\cM_1(\R^d,\abs{\cdot}\wedge1)$,
\begin{equation} \label{eq:equiv_bddLip}
	\frac{1}{2}\mathbf{d}(\mu,\nu)\leq \bar{\cW}_1(\mu,\nu)\leq 2\mathbf{d}(\mu,\nu).
\end{equation}
\end{rem}

Unless specifically stated, $\cM(\R^d)$ and $\cM_p(\R^d)$ for all $p\geq 1$ are equipped with the metric $\mathbf{d}$, while $\cP_p(\R^d)$ is equipped with the Wasserstein distance $\cW_p$.
Under certain circumstances, the existence and regularity of density are required for a finite measure as follows,
which is used in both the prior work \cite{hambly2025optimal} and Section~\ref{append_flow} in this paper in order to ensure the well-posedness of nonlinear Fokker--Planck equations.
\begin{defn}\label{density_regularity}
    For $\theta>0$, denote by $\cM^\theta(\R^d)$ the set of all $\mu\in\cM(\R^d)$
    with a density $\varrho$ such that
    \[
        \int_{\R^d}\varrho(x)^2\exp\Big(\theta\sqrt{1+\abs{x}^2}\Big)\,\d x<\infty.
    \]
    Moreover, we define 
    \[\cM_p^\theta(\R^d):=\cM_p(\R^d)\cap\cM^\theta(\R^d).\]
\end{defn}

 \paragraph{Functionals on the space of measures}

	For functionals on $\cM_2(\R^d)$, we follow \cite[Definitions 5.22, 5.43]{carmona2018probabilistic} to introduce give the linear and Intrinsic derivatives.

\begin{defn}[Linear and Intrinsic derivatives]\label{derivatives}
    For a functional $F:\cM_2(\R^d)\to\R$, we say $\frac{\delta F}{\delta\mu}\,(\text{or }\delta_\mu F):\cM_2(\R^d)\times\R^d\to\R$ is a linear (functional) derivative of $F$, if $\delta_\mu F(\mu,\cdot)$ has at most quadratic growth, uniformly in $\mu$, and
    \[F(\mu)-F(\nu)=\int_0^1\big\langle\delta_\mu F(\lambda\mu+(1-\lambda)\nu,\cdot),\mu-\nu\big\rangle\,\d\lambda\]
    holds for all $\mu,\nu\in\cM_2(\R^d)$.
    Moreover, we say $D_\mu F:\cM_2(\R^d)\times\R^d\to\R^d$ is an intrinsic derivative of $F$, if there exists a linear derivative $\delta_\mu F$ satisfying
    \[D_\mu F(\mu,x)=D_x\delta_\mu F(\mu,x)=\big(\partial_{x_i}\delta_\mu F(\mu,x)\big)_{i=1}^d\]
    for all $\mu\in\cM_2(\R^d),\,x=(x_1,...,x_d)\in\R^d$.

    Inductively, suppose $\delta_\mu F$ exists, and we say $\frac{\delta^2 F}{\delta\mu^2}\,(\text{or }\delta_\mu^2 F):\cM_2(\R^d)\times\R^d\times\R^d\to\R$ is a second-order linear derivative of $F$,
    if $\delta_\mu^2 F(\mu,\cdot,\cdot)$ has at most quadratic growth, uniformly in $\mu$, and $\delta_\mu^2 F(\cdot,x,\cdot)$ is a linear derivative of $\delta_\mu F(\cdot,x)$
    holds for all $x\in\R^d$.
    Moreover, we say $D_\mu^2 F:\cM_2(\R^d)\times\R^d\times\R^d\to\R^{d\times d}$
    if some $\delta_\mu^2 F$ exists and satisfies
    \[D_\mu^2 F(\mu,x,x^\prime)=D_x\prime D_x\delta_\mu^2 F(\mu,x,x^\prime)=\big(\partial_{x_i^\prime}\partial_{x_j}\delta_\mu^2 F(\mu,x,x^\prime)\big)_{i,j=1}^d\]
    for all $\mu\in\cM_2(\R^d),\,x=(x_1,...,x_d),x^\prime=(x_1^\prime,...,x_d^\prime)\in\R^d$.
\end{defn}
On the basis of the above definition, we can consequently define differentiability of functionals.
\begin{defn}[$\cC^2$ and $\cC_b^2$ functionals]\label{C^2} 
    For a functional $F:\cM_2(\R^d)\to\R$, we say $F\in\cC^2(\cM_2(\R^d))$, if
    \begin{itemize}
        \item derivatives $\frac{\delta F}{\delta\mu},\,\frac{\delta^2 F}{\delta\mu^2}$ and $D_{\mu}F,\,D_\mu^2 F$ exist in the sense of Definition~\ref{derivatives}, and are globally continuous;
        \item the Hessian matrix $D_x^2\frac{\delta F}{\delta\mu}:\cM_2(\R^d)\times\R^d\to\R^{d\times d}$ exists and is continuous.
    \end{itemize}
    Furthermore, we say $F\in\cC_b^2(\cM_2(\R^d))$, if $F\in\cC^2(\cM_2(\R^d))$ and all its derivatives and the Hessian matrix mentioned above are bounded.
\end{defn}

\begin{defn}[$\cC^{1,2}$ functional]\label{C^1,2}
    For a functional $F:[0,T]\times\cM_2(\R^d)\to\R$, we say $F\in\cC_b^{1,2}\big([0,T]\times\cM_2(\R^d)\big)$, if
    \begin{itemize}
        \item for each $t\in[0,T]$, $F(t,\cdot)\in\cC^2(\cM_2(\R^d))$;
        \item for each $\mu\in\cM_2(\R^d)$, $F(\cdot,\mu)\in\cC^1([0,T])$;
        \item all derivatives of $F$, with respect to $t$ of first order, and with respect to $\mu$ of first and second orders, including $D_x^2\delta_\mu F$ (see Definition~\ref{C^2}), are globally continuous on $[0,T]\times\cM_2(\R^d)$ (or $[0,T]\times\cM_2(\R^d)\times\R^d$, $[0,T]\times\cM_2(\R^d)\times\R^d\times\R^d$);
        \item $\delta_\mu F(t,\mu,\cdot),D_\mu F(t,\mu,\cdot)$, $\delta_\mu^2 F(t,\mu,\cdot,\cdot),D_\mu^2 F(t,\mu,\cdot,\cdot)$ and $D_x^2\delta_\mu F(t,\mu,\cdot)$ all have at most quadratic growth, uniformly in $(t,\mu)\in [0,T]\times\cM_2(\R^d)$.
    \end{itemize}
\end{defn}

\begin{rem}\label{remark_E_projection}
    In fact, \eqref{SDE_E_MeanField} and \eqref{SDE_E_particle} are simplified versions of the branching dynamics in \cite{claisse2024mckean},
    and we observe that $\mu_t=\Psi m_t ,\,t\in[0,T]$ for the projection 
    \begin{equation}\label{E_projection}
        \begin{aligned}
            \Psi:\cP_1(E)\,&\to \cM(\R^d)\\
            m\,&\mapsto\mu=\Psi m\quad\text{s.t.}\\
            \langle f,\mu\rangle=\,&\int_E \langle f,e\rangle\,m(\d e),\quad\forall f\in\cC_b(\R^d).
        \end{aligned}
    \end{equation}
    It is clear that $\Phi$ is surjective, and
    by Kantorovich duality (cf.~\cite{villani2008optimal}) we have
    \[\mathbf{d}(\Psi m,\Psi m^\prime)\leq\cW_1(m,m^\prime)\]
    for all $m,m^\prime\in\cP_1(E)$.
\end{rem}

\begin{rem}
    Definitions~\ref{derivatives}, \ref{C^2} and \ref{C^1,2} can be naturally expanded to encompass differentiable functionals on convex subsets of $\cM_2(\R^d)$ such as $\cM_p(\R^d)$ ($p\geq2$) and $\cP_2(\R^d)$, without changing the quadratic growth requirement.
\end{rem}

\begin{rem}
   Under certain regularity conditions of $F:\cP_2(\R^d)\to\R$, intrinsic derivatives of $F$ given by Definition~\ref{derivatives} are equivalent to the ``L-derivative'' derived from Fr{\'e}chet derivatives in \cite[Definitions 5.22]{carmona2018probabilistic} (cf.~\cite[Propositions 5.44, 5.48]{carmona2018probabilistic}).
\end{rem}

\subsection{The McKean--Vlasov branching diffusion}
\label{subsec:MKV_branch_diffus}

Let $T$ and $\bar{\gamma}$ be positive constants. 
We consider the filtered probability space $(\Omega,\cF,(\cF_t)_{t\geq 0},\dbP)$, equipped with a family of $\R^d$-valued standard Brownian motions $(W^k)_{k\in\mathbb{K}}$ and Poisson random measures $\big(Q^k(\d s,\d z)\big)_{k\in\mathbb{K}}$ on $[0,T]\times [0,\bar{\gamma}]\times[0,1]$ with Lebesgue intensity measure $\d s\d z$.
The random elements $(W^k)_{k\in\mathbb{K}}$ and $\big(Q^k(\d s,\d z)\big)_{k\in\mathbb{K}}$ are mutually independent.

\vspace{0.5em}

Let us first introduce our McKean--Vlasov branching diffusion process in a descriptive way, and then define it rigorously by the SDE \eqref{SDE_E_MeanField}.
The branching diffusion will be defined as an $E$-valued progressively measurable process $\bar{Z}=(\bar{Z}_t)_{t\in[0,T]}$, with
\begin{equation}\label{solution}
    \bar{Z}_t= \sum_{k\in K_t}\delta_{(k,X_t^k)}(\omega)\in E,\quad\forall t\in [0,T],
\end{equation}
where $K_t$ denotes the set of all labels of particles alive at $t$, and $X_t^k$ is the value of particle $k$ at time $t$.
For each particle $k\in K_t$, $X^k$ evolves as a diffusion governed by
\begin{equation}\label{branching_diffusion}
    \d X_t^k=b(t,X_t^k,\mu_t)\,\d t+\sigma(t,X_t^k,\mu_t)\,\d W_t^k,
\end{equation}
with coefficients $(b,\sigma):[0,T]\times\R^d\times\cM_2(\R^d)\to\R^d\times\R^{d\times d}$ and the environment measure $\mu_t\in\cM_2(\R^d)$, namely,
\begin{equation}\label{eq:environment}
    \langle f,\mu_t\rangle:=\Expect\Big[\sum\limits_{k\in K_t}f(X_t^k)\Big]
\end{equation}
for each bounded measurable function $f$.
Let us denote the infinitesimal generator by $\cL$, \ie,~for each $f\in\cC^2(\R^d),\,(t,x,\mu)\in[0,T]\times\R^2\times\cM_2(\R^d)$,
\begin{equation}\label{infititesimal}
    \cL f(t,x,\mu):=b(t,x,\mu)\cdot D_xf(x)+\frac{1}{2}\trace(D_x^2 f(x)\sigma\sigma^\transpose(t,x,\mu)).
\end{equation}

Let $(\gamma,(p_l)_{l\in\dbN}):[0,T]\times\R^d\times\cM_2(\R^d)\to [0,\bar{\gamma}]\times [0,1]^{\dbN}$ be the death rate and probability mass function of the progeny distribution for \eqref{solution}, where $\sum_{l\in\dbN} p_l\equiv1$.
Hence, $(p_l)_{l\in\dbN}$ naturally induces a partition of $[0,1]$, namely
\[I_l:=\Big[\sum_{i=0}^{l-1}p_l,\sum_{i=0}^{l}p_l\Big),\quad\forall l\in\dbN.\]
Denote by $S_k$ the birth time of each particle $k\in\mathbb{K}$, where in particular $S_k=0$ for $k\in K_0$.
After birth, $k$ evolves as the diffusion \eqref{branching_diffusion}, until its death time
\[T_k:=\inf\Big\{t>S_k:\,Q^k\big(\{t\}\times [0,\gamma(t,X_t^k,\mu_t)]\times[0,1]\big)=1\Big\}\]
Let $U^k$ be a random variable uniformly distributed on $[0,1]$, which satisfies
\[Q^k\big(\{T_k\}\times[0,\gamma(t,X_t^k,\mu_t)]\times\{U^k\}\big)=1.\]
At $T_k$, the particle $k$ gives birth to $l$ offspring particles after death, if $U^k\in I_l(T_k,X_{T_k}^k,\mu_{T_k})$.
Label these offspring particles as $\{k1,k2,...,kl\}$ (or $\emptyset$ if $l=0$),
and we have
\[K_{T_k}=(K_{T_{k-}}\setminus\{k\})\cup \{k1,k2,...,kl\}.\]
Moreover, each offspring starts from its ancestor position, \ie
\[X_{S_{ki}}^{ki}=X_{T_{k-}}^k,\quad \forall i=1,2,...,l,\]
and we extend its trajectory by $X_t^{ki}:=X_t^k$ for all $t\in[0,T_k)$.

\vspace{0.5em}

More rigorously, the McKean--Vlasov branching diffusion $\bar{Z}$ described above can be defined as solution to the SDE:
\begin{equation}\label{SDE_E_MeanField}
    \begin{aligned}
        &\langle \phi,\bar{Z}_t\rangle=\,\langle \phi,\bar{Z}_0\rangle+\int_0^t\sum_{k\in K_s}\cL\phi^k(s,X_s^k,\mu_s)\,\d s+\int_0^t\sum_{k\in K_s}D\phi^k(X_s^k)\cdot\sigma(s,X_s^k,\mu_s)\,\d W_s^k\\
        &\quad+\int_{[0,t]\times[0,\bar{\gamma}]\times[0,1]}\sum_{k\in K_{s-}}\sum_{l\in\dbN}\Big(\sum_{i=1}^l \phi^{ki}-\phi^k\Big)(X_s^k)\mathbbm{1}_{[0,\gamma(s,X_s^k,\mu_s)]\times I_l(s,X_s^k,\mu_s)}(z)\,Q^k(\d s,\d z)
    \end{aligned}
\end{equation}
for each $\phi=(\phi^k)_{k\in\mathbb{K}}$ such that $\phi^k\in\cC_b^2(\R^d)$ for all $k\in\mathbb{K}$, where $(\mu_s)_{s \in [0,T]}$ is given by \eqref{eq:environment}.

\paragraph{Finite-population branching process}
At the same time, we can analogously define a finite-population particle system $\{Z^i=(Z^i_t)_{t\in[0,T]}\}_{i=1}^N$ on the filtered probability space
$(\Omega^N,\cF^N,(\cF_t)_{t\geq 0}^N,\dbP^N)$
by
\begin{equation}\label{solution_particle}
    Z_t^i=Z_t^i(\omega)=\sum_{k\in K_t^i}\delta_{(k,X_t^{i,k})}(\omega)\in E,
\end{equation}
which satisfies
\begin{equation}\label{SDE_E_particle}
    \begin{aligned}
        &\langle \phi,Z_t^i\rangle=\,\langle \phi,Z_0^i\rangle+\int_0^t\sum_{k\in K_s^i}\cL\phi^k(s,X_s^{i,k},\mu_s^N)\,\d s+\int_0^t\sum_{k\in K_s^i}D\phi^k(X_s^{i,k})\cdot\sigma(s,X_s^{i,k},\mu_s^N)\,\d W_s^{i,k}\\
        &\quad+\int_{[0,t]\times[0,\bar{\gamma}]\times[0,1]}\sum_{k\in K_{s-}^i}\sum_{l\in\dbN}\Big(\sum_{j=1}^l \phi^{kj}-\phi^k\Big)(X_s^{i,k})\mathbbm{1}_{[0,\gamma(s,X_s^{i,k},\mu_{s-}^N)]\times I_l(s,X_s^{i,k},\mu_{s-}^N)}(z)\,Q^{i,k}(\d s,\d z),
    \end{aligned}
\end{equation}
for each $i=1,...,N$, $\phi=(\phi^k)_{k\in\mathbb{K}}$ such that $\phi^k\in\cC_b^2(\R^d)$ for all $k\in\mathbb{K}$.
Here $\{Q^{i,k},W^{i,k}\}_{1\leq i\leq N,k\in \mathbb{K}}$ is a mutually independent group of Poisson random measures $\{Q^{i,k}\}_{1\leq i\leq N,k\in \mathbb{K}}$ with Lebesgue intensity measures on $[0,T]\times [0,\bar{\gamma}]\times[0,1]$ and standard Brownian motions $\{W^{i,k}\}_{1\leq i\leq N,k\in \mathbb{K}}$,
while $\mu_t^N$ denotes the empirical measure for the particle system \eqref{solution_particle}, namely,
\begin{equation}\label{eq:empirical}
    \langle f,\mu_t^N\rangle:=\frac{1}{N}\sum\limits_{i=1}^N\sum\limits_{k\in K_t^i}f(X_t^{i,k})
\end{equation}
for each bounded measurable function $f$.

\vspace{0.5em}

Let us first introduce some technical conditions on the coefficients to ensures the well-posedness of $\bar{Z}$ and $\{Z^i\}_{i=1}^N$ in \eqref{SDE_E_MeanField} and \eqref{SDE_E_particle}.

\begin{assu}\label{assum:coefficients}
    Let $(b,\sigma,\gamma,(p_l)_{l\in\N}):[0,T]\times\R^d\times\cM_2(\R^d)\to\R^d\times\R^{d\times d}\times[0,\bar{\gamma}]\times[0,1]^{\N}$ be the coefficients of the branching diffusion defined in Subsection~\ref{preliminary}, and denote by $c(t,x,\mu):=\gamma(\sum_l lp_l-1)(t,x,\mu)$.
    Suppose it holds:
    
    \emph{(i)} there exists a constant $M\in(0,\infty)$ such that
    \[\norm{b}_{\infty},\,\norm{\sigma}_{\infty},\,\norm{\sum\limits_{l\in\dbN} l p_l}_{\infty}\leq M;\]
    
    \emph{(ii)} there exists a constant $L\in (0,\infty)$ such that 
    \[
    \begin{aligned}
        &\quad\abs{b(t,x,\mu)-b(t,y,\nu)},\,\norm{\sigma (t,x,\mu)-\sigma (t,y,\nu)},\,\abs{c(t,x,\mu)-c(t,y,\nu)} \\
        &\leq\,L\big(\abs{x-y}+\mathbf{d}(\mu,\nu)\big),
    \end{aligned}\]
    for all $(t,x,\mu)\in[0,T]\times\R^d\times\cM_2(\R^d)$.
\end{assu}

\begin{assu}\label{assum:initial}
    	The initial conditions $\bar{Z}_0$ and $\{Z_0^i\}_{i=1}^N$ for \eqref{SDE_E_MeanField} and \eqref{SDE_E_particle} satisfy that
	$\mu_0\in\cM_2(\R^d)$, where $\mu_0$ is defined by \eqref{eq:environment}.
	Moreover, $\bar{Z}_0, Z_0^1,...,Z_0^N$ are i.i.d. $E$-valued random variables.
\end{assu}

\begin{prop}
	Let Assumptions~\ref{assum:coefficients} and \ref{assum:initial} hold true.
	Then both branching SDE  \eqref{SDE_E_MeanField} and \eqref{SDE_E_particle} have a unique strong solution denoted respectively by $\bar Z$ and $\{ Z^i \}^N_{i=1}$.
\end{prop}

The well-posedness of McKean--Vlasov branching SDE \eqref{SDE_E_MeanField}  can be ensured by ~\cite[Theorems 2.3, 2.8; Remark 2.8]{claisse2024mckean}.
As for finite-population branching SDE \eqref{SDE_E_particle}, since $(x, x_1, \cdots, x_n) \mapsto (b, \sigma) (x, \sum_{i=1}^n \delta_{x_i})$ is clearly Lipschitz under Assumption \ref{assum:coefficients}, one can use the constructive way to obtain the unique process $\{ Z^i \}^N_{i=1}$ as strong solution to \eqref{SDE_E_particle}.

\begin{rem}
    Under Assumptions~\ref{assum:coefficients} and \ref{assum:initial}, $\mu_t,\mu_t^N\in\cM_2(\R^d),\,t\in[0,T]$ by \cite[Lemma 2.5]{claisse2025optimal},
    and it similarly holds, for each $p\geq 2$, 
    that $\mu_t,\mu_t^N\in\cM_p(\R^d),\,t\in[0,T]$ if $\mu_0\in\cM_p(\R^d)$.
    On the other hand, by \cite[Lemma 3.1]{claisse2024mckean} we have
    \begin{equation}\label{mass_growth}
        \sup_{r\in[t,s]}\Expect[\# K_r]\leq \Expect[\# K_t]\exp\big(\bar{\gamma}M(s-t)\big)
    \end{equation}
    for each $0\leq t<s\leq T$, and thus $m_t\in \cP_1(E)$ for all $t\in[0,T]$.
\end{rem}

\subsection{Quantitative weak propagation of chaos}\label{main_theorems}

	Our main objective is to provide a convergence rate for the approximation of $\mu_T$ in \eqref{eq:environment} with empirical measure $\mu_T^N$ defined in \eqref{eq:empirical}.

\vspace{0.5em}

	To achieve this, let us first provide an It\^o's formula for functional along the flow of measures $(\mu_t)_{t\in[0,T]}$ as well as that of the empirical measures $(\mu_t^N)_{t\in[0,T]}$.
	Notice that the measures $(\mu_t)_{t \in [0,T]}$ are not necessary probability measures, so that the corresponding It\^o's formula consists in an extension of the It\^o's formula for semi-martingale marginal distributions in e.g. \cite{10.1214/15-AOP1076}.
	
\begin{thm}[It\^o's formula]\label{ito_formula}
    Let Assumptions \ref{assum:coefficients} and \ref{assum:initial} hold, and $\mu_t$ be defined by \eqref{eq:environment} and $\mu_t^N$ be the empirical measure in \eqref{eq:empirical}.
    For a functional $F\in\cC^{1,2}\big([0,T]\times\cM_2(\R^d)\big)$ (recall Definitions~\ref{C^2} and \ref{C^1,2}), we have (recall \eqref{infititesimal}), for $t\in[0,T]$,
    \begin{align}
        \label{ito_environment}&\begin{aligned}
            \d F(t,\mu_t)=\,\bigg(\partial_t F(t,\mu_t)+\Big\langle \cL \frac{\delta F}{\delta \mu}+\gamma\Big(\sum\limits_l lp_l-1\Big)\frac{\delta F}{\delta \mu} ,\mu_t\Big\rangle\bigg)\,\d t
        \end{aligned}\\
        \label{ito_empirical}&\begin{aligned}
        \d F(t,\mu_t^N)=& \,\bigg(\partial_t F(t,\mu_t^N)+\Big\langle \cL \frac{\delta F}{\delta \mu},\mu_t^N\Big\rangle\bigg)\,\d t\\
            &\,+\frac{\d t}{2N}\int_{\R^d}\trace\big(D_{\mu}^2 F(t,\mu_t^N,x,x)\sigma\sigma^\transpose(t,x,\mu_t^N)\big)\,\mu_t^N(\d x)\\
            &\,+N\,\d t\int_{\R^d}\gamma\sum\limits_{l}p_l(t,x,\mu_t^N)\Big(F\big(t,\mu_t^N+\frac{l-1}{N}\delta_x\big)-F\big(t,\mu_t^N\big)\Big)\,\mu_t^N(\d x)\\
            &\,+\d \mathbf{M}_t,
        \end{aligned}
    \end{align}
    where $\mathbf{M}$ is a martingale.
    Here we write, for simplicity,
    \[\langle g,\nu\rangle:=\int_{\R^d} g(t,x,\nu)\,\nu(\d x)\]
    for $(t,\nu)\in[0,T]\times\cM_2(\R^d)$ and $g:[0,T]\times\R^d\times\cM_2(\R^d)\to\R$.
\end{thm}
	The proof will be completed in Section~\ref{append_ito} in Appendix.

\vspace{0.5em}

Under Assumption \ref{assum:coefficients}, we recall (see Remark \ref{remark_E_projection}) 
\[m_t=\text{Law}(\bar{Z}_t)\in \Psi^{-1}(\cM_2(\R^d))\subset\cP_1(E),\quad t\in[0,T]\]
for each $t\in[0,T]$, which naturally induces
\begin{equation}\label{law_flow}
    m_s^{t,m}:=\text{Law}(\bar{Z}_s|\bar{Z}_t\sim m)
\end{equation}
for each $0\leq t\leq s\leq T,\,m\in\Psi^{-1}(\cM_2(\R^d))$, and satisfies the flow property, namely,
\[m_s=m_s^{t,m_t},\, m_s^{r,m_r^{t,m}}=m_s^{t,m} \quad 0\leq t\leq r\leq s\leq T,\, m\in\Psi^{-1}(\cM_2(\R^d)).\]
For the environment measure $\mu_t=\Psi m_t$, however, such flow is non-trivial.
For instance, for $m,\hat{m}\in\cP_1(E)$ such that $\Psi m=\Psi\hat{m}\in\cM_2(\R^d)$, we do not typically have
\[\Psi m_s^{t,m}=\Psi m_s^{t,\hat{m}}.\]

To recover the flow property on $\mu_t = \Psi m_t$, we will rely on a uniqueness result of the solution to a nonlinear Fokker--Planck equation (c.f.~Section~\ref{append_flow}),
which requires the following non-degeneracy condition of the diffusion coefficient.
\begin{assu}\label{assum:nondegen}
    Let $\sigma:[0,T]\times\R^d\times\cM_2(\R^d)\to\R^{d\times d}$ be the coefficient of the branching diffusion as defined in Subsection~\ref{preliminary}.
    Suppose $\sigma$ is uniformly non-degenerated in the sense that
    \[\sigma\sigma^\transpose (t,x,\mu)\geq \varepsilon_0 I_d,\quad\forall (t,x,\mu)\in [0,T]\times\R^d\times\cM_2(\R^d)\]
    for some constant $\varepsilon_0>0$.
\end{assu}

The proof of weak propagation of chaos relies on analysis of a Backward Kolmogorov equation \eqref{eq:bKol} as well as the functional derivatives of the value function.
Hence, we impose certain regularity on coefficients $b,\sigma,c$ and the test functional $G$ as follows.
\begin{assu}\label{assum:differentiability}
    Let $b=(b_i)_{i=1}^d$, $\sigma=(\sigma_{ij})_{i,j=1}^d$ and $c$ be the coefficients defined above (recall Assumption~\ref{assum:coefficients}).
    For each $1\leq i\leq j\leq d$, suppose it holds: 

    \emph{(i)} $b_i(t,x,\cdot),\,\sigma_{ij}(t,x,\cdot),\, c(t,x,\cdot)$
    belong to $\cC_b^{1,2}$ for all $(t,x)\in[0,T]\times\R^d$;

    \emph{(ii)} $b_i(t,\cdot,\mu),\,\sigma_{ij}(t,\cdot,\mu),\, c(t,\cdot,\mu)$ belong to $\cC_b^{2}(\R^d)$ for all $(t,\mu)\in[0,T]\times\cM_2(\R^d)$;

    \emph{(iii)} all derivatives of $b_i,\,\sigma_{ij},\, c$ up to second-order are bounded and Lipschitz continuous on $[0,T]\times\R^d\times\cM_2(\R^d)$ (or $[0,T]\times\R^d\times\cM_2(\R^d)\times\R^d$, $[0,T]\times\R^d\times\cM_2(\R^d)\times\R^d\times\R^d$).

    We also assume the test functional $G:\cM_2(\R^d)\to\R$ is $\cC_b^2$, 
    with its derivatives $\frac{\delta G}{\delta\mu},\,\frac{\delta^2 G}{\delta\mu^2}$, $D_{\mu}G,\,D_\mu^2 G$ and $D_x^2\frac{\delta G}{\delta\mu}$ globally Lipschitz continuous.
\end{assu}

Finally, we state the main theorem of quantitative weak propagation of chaos for branching diffusions.
\begin{thm}[Weak propagation of chaos]\label{thm_WPoC}
    Suppose Assumptions \ref{assum:coefficients}, \ref{assum:initial}, \ref{assum:nondegen} and \ref{assum:differentiability} hold,
    and the initial value of \eqref{solution} satisfies (recall Definition~\ref{density_regularity})
    \[\Expect[(\# K_0)^2]<\infty,\quad \mu_0\in\cM_{2+\delta}^\theta(\R^d)\]
    for some $\delta,\theta>0$.
    Then the empirical measure $\mu_T^N$ in \eqref{eq:empirical} converges weakly to the environment measure $\mu_T$ \eqref{eq:environment} as $N\to\infty$, in the sense that 
    \[\abs{G(\mu_T)-\Expect[G(\mu_T^N)]}\leq C\frac{1}{N},\]
    where $C>0$ is a constant independent of $N$.
\end{thm}

\begin{rem}
    The $(2+\delta)$th moment of $\mu_0$ is required for regularity of the value function (see Lemma~\ref{lifting_inverse_continuity}),
    while density yields flow property of $\mu_t$ (refer to Section~\ref{append_flow}).
    The second moment of $\#K_0$ is used in the estimate of Lemma~\ref{lem_initial}.
\end{rem}

\section{Proofs}\label{proofs}
In Subsection~\ref{subsec:bkol}, we introduce a non-branching McKean-Vlasov diffusion process, whose marginal distribution $\rho_t$ induces the marginal measure $\tilde \mu_t$ which satisfies the same Fokker--Planck equation as that of the branching diffusion $\mu_t$, so that the conditional value of $G(\tilde \mu_T)$ and that of $G(\mu_T)$ satisfy the same Kolmogorov backward master equation \eqref{eq:bKol}.
We then introduce the value function $U(t,\mu)$ as conditional value of $G(\tilde \mu_T)$, and study its regularity based on the non-branching McKean-Vlasov diffusion in Subsection~\ref{proof_value_regularity}.
Finally, in Subsection~\ref{proof_main}, we complete proof of Theorem~\ref{thm_WPoC}.

\subsection{A non-branching diffusion process and  Kolmogorov backward equation}\label{subsec:bkol}

Following the classical approach such as in \cite{10.1214/19-EJP298, chassagneux2022weak}, etc., we will introduce a value function $U$ as conditional value of $U(\mu_T)$,
which  serves as key step to estimate $\abs{G(\mu_T)-\Expect[G(\mu_T^N)]}$ (refer to \eqref{functional_taylor}).
However, the regularity analysis of the branching SDE is far more intricate than that of non-branching SDEs, we therefore introduce a $\R^{d+1}$-valued non-branching diffusion $(Y, Z)$, which can be considered as a lifted process of the branching diffusion $\bar Z$. 
Let us also mention that a similar lifting technique has also been used in \cite{wang2024mean}.

\vspace{0.5em}

Let us consider the diffusion $(Y,Z)={(Y_t,Z_t)_{t\in[0,T]}}$, governed by the following SDE (recall $c(\cdot)=\gamma(\sum_l lp_l-1)(\cdot)$ in Assumption \ref{assum:coefficients})
\begin{equation}\label{lifted_MeanField}
    \left\{
    \begin{aligned}
        \d Y_t&=b(t,Y_t,\tilde{\mu}_t)\,\d t+\sigma(t,Y_t,\tilde{\mu}_t)\,\d W_t,\\
        \d Z_t&=c(t,Y_t,\tilde{\mu}_t)Z_t\,\d t,
    \end{aligned}
    \right.
\end{equation}
where $Z_t\in[0,\infty)~\text{a.s.}$,
and $\tilde{\mu}_t\in\cM_2(\R^d)$ satisfies, for all bounded measurable functions $\varphi:\R^d\to\R$,
\[\langle \varphi,\tilde{\mu}_t\rangle=\Expect[\varphi(Y_t)Z_t].\]
Let
$$\rho_t:=\text{Law}(Y_t,Z_t)\in\cP(\R^d\times [0,\infty))$$
be the marginal distribution of $(Y,Z)$, so that 
\begin{equation}\label{dynamics_dual}
    \langle \varphi,\tilde{\mu}_t\rangle=\int_{\R^d\times\R^1}\varphi(y)z\,\rho_t(\d y\d z).
\end{equation}
In other words, the finite measure $\tilde{\mu}_t$ is a ``projection'' of the probability measure $\rho_t$.
By $Z_t\geq0$ and Assumption~\ref{assum:coefficients}, it holds, for all $0\leq t\leq s\leq T$, that
\[Z_t\exp\big(-\bar{\gamma}M(s-t)\big)\leq Z_s\leq Z_t\exp\big(\bar{\gamma}M(s-t)\big),\]
which implies $\rho_t\in\cP(\R^d\times[0,C\exp(\bar{\gamma}MT)]),\,t\in[0,T]$ if $\rho_0\in\cP(\R^d\times[0,C])$ for some $C\in[0,\infty)$.

\begin{lem} \label{lem:FKP}
	Both flows of measures $(\mu_t)_{t \in [0,T]}$ in \eqref{eq:environment} and $(\tilde \mu_t)_{t \in [0,T]}$ in \eqref{dynamics_dual} satisfy the same nonlinear Fokker--Planck equation:
	for each $t \le s \le T$ and $f\in\cC^{1,2}_b([t,T]\times\R^d)$, with $f_r(\cdot):=f(r,\cdot)$,
	\begin{equation}\label{F-P_duality}
	\langle f_s,\mu_s\rangle
		=\langle f_t,\mu\rangle+\int_t^s\Big\langle \partial_r f_r(\cdot)+\Big(\cL f_r+\gamma\Big(\sum\limits_l lp_l-1\Big)f_r\Big)(r,\cdot,\mu_r),\mu_r\Big\rangle\,\d r.
	\end{equation}	
\end{lem}
The proof for $(\mu_t)_{t \in [0,T]}$ is a direct consequence of \eqref{eq:environment}-\eqref{SDE_E_MeanField}, see also Appendix of \cite{claisse2025optimal}.
The proof of $(\tilde \mu_t)_{t \in [0,T]}$ is similar, which is omitted.

\vspace{0.5em}

Now we consider the lifting operator
\begin{align*}
    \cT:\cC(\R^d)&\to \cC\big(\R^d\times[0,\infty)\big)\\
    \varphi&\mapsto\tilde{\varphi}=:\cT\varphi\quad\text{s.t.}\\
    \tilde{\varphi}(y,z)&=\varphi(y)z\quad\text{for~}(y,z)\in\R^d\times[0,\infty),
\end{align*}
and its ``dual'' in the sense of the projection
\begin{align*}
    \cT^*:\cP_2\big(\R^d\times[0,\infty)\big)&\to\cM(\R^d)\\
    \rho&\mapsto\mu=:\cT^*\rho\quad\text{s.t.}\\
    \mu(\d y)&=\int_0^\infty z\,\rho(\d y\d z)\ \text{for~}(y,z)\in\R^d\times[0,\infty).
\end{align*}
It holds clearly that
\begin{equation}\label{function_dual}
    \langle\varphi,\cT^*\rho\rangle=\langle\cT\varphi,\rho\rangle
\end{equation} 
for all $\varphi\in\cC_b(\R^d),\,\rho\in\cP_2(\R^d\times[0,\infty))$, and that $\tilde{\mu}_t=\cT^*\rho_t$ for all $t\in[0,T]$.
Furthermore, if $\rho\in\cP_2(\R^d\times[0,C])$ for some $C\in[0,\infty)$,
then $\cT^*\rho\in\cM_2(\R^d)$ and $(\cT^*\rho)(\R^d)\leq C$.
The following lemma shows the local Lipschitz continuity of $\cT^*$.
\begin{lem}\label{Lipschitz_lifting}
    For each $C\in[0,\infty)$, it holds that (recall Definition~\ref{wasserstein})
    \[\mathbf{d}(\cT^*\rho^1,\cT^*\rho^2)\leq (C+1)\cW_1(\rho^1,\rho^2)\leq (C+1)\cW_2(\rho^1,\rho^2)\]
    for all $\rho^1,\rho^2\in\cP_2\big(\R^d\times[0,C]\big)$.
\end{lem}
\begin{proof}
    By Definition~\ref{wasserstein}, we have
    \[\mathbf{d}(\cT^*\rho^1,\cT^*\rho^2)=\sup_{\norm{f}_{\infty},\,\norm{f}_{\text{Lip}}\leq 1}\langle\cT f,\rho^1-\rho^2\rangle.\]
    Note that
    \[\norm{\cT f}_{\text{Lip}(\R^d\times[0,C])}\leq \norm{f}_{\infty}+C\norm{f}_{\text{Lip}},\]
    and hence we have, by Kantorovich duality (cf.~\cite{villani2008optimal}),
    \[\mathbf{d}(\cT^*\rho^1,\cT^*\rho^2)\leq (C+1)\cW_1(\rho^1,\rho^2).\]
\qed
\end{proof}
By Lemma~\ref{Lipschitz_lifting}, one can further establish the well-posedness of the non-branching SDE \eqref{lifted_MeanField} (refer to Section~\ref{lifted_wellposedness_proof} for proof).
\begin{lem}\label{lifted_wellposedness}  
	Suppose that Assumption~\ref{assum:coefficients} holds,
    and the initial condition $(Y_0, Z_0) \sim \rho_0\in\cP_2(\R^d\times[0,C])$ for some constant $C\in[0,\infty)$.
    Then there exists a unique strong solution 
    to the initial value problem of \eqref{lifted_MeanField}.
\end{lem}
\begin{rem}
    We observe that the marginal distribution $\rho_t$ and the flow $\rho_s^{t,\rho}$ below can take value throughout the space $\cup_{C\geq0}\cP_2(\R^d\times[0,C])$.
    For simplicity, we still write $\cP_2(\R^d\times[0,C])$, where $C$ represents a positive constant whose value can vary through lines.
\end{rem}

Consequently, \eqref{lifted_MeanField} induces the flow (cf.~\cite[Proposition 1, Section 2.2]{chaintron2022propagation})
\[\rho_{s}^{t,\rho}:=\text{Law}\Big((Y_s,Z_s)\big| (Y_t,Z_t)\sim\rho\Big),\quad 0\leq t\leq s\leq T,\,\rho\in\cP_2(\R^d\times[0,C]),\]
which satisfies
\[\rho_s^{t,\rho_t}=\rho_s,\,\rho_s^{r,\rho_r^{t,\rho}}=\rho_s^{t,\rho},\quad 0 \leq t\leq r\leq s\leq T,\,\rho\in\cP_2(\R^d\times[0,C]).\]
Moreover, the flow $\rho_\cdot^{t,\cdot}$ satisfies the following property (refer to Subsection~\ref{proof:mollify} for proof).
\begin{lem}\label{lem:mollify}
    Let Assumptions~\ref{assum:coefficients} and \ref{assum:nondegen} hold.
    For all $\rho,\hat{\rho}\in\cP_2(\R^d\times[0,C])$ such that $\cT^*\rho=\cT^*\hat{\rho}$, it holds that
    \[\cT^*\rho_s^{t,\rho}=\cT^*\rho_s^{t,\hat{\rho}}\]
    for all $0\leq t\leq s\leq T$.
\end{lem}

On the other hand, we define an ``inverse'' mapping $\Phi$ of $\cT^*$, given by
\begin{equation}\label{lifting_inverse}
    \Phi\mu:=\bar{\mu}\otimes\delta_{\mu(\R^d)},\quad\mu\in\cM_2(\R^d).
\end{equation}
Here $\bar{\mu}$ denotes the normalized probability measure of some $\mu\in\cM(\R^d)$ such that $\bar{\mu}=\mu/\mu(\R^d)$ if $\mu\neq0$, and $\delta_0$ if otherwise.
We observe, for each $\mu\in\cM_2(\R^d)$,
that $\rho=\Phi\mu\in\cP_2(\R^d\times[0,C])$ for $C=\mu(\R^d)$ and $\cT^*\rho=\mu$,
which also implies $\cT^*$ maps $\cup_{C\geq0}\cP_2(\R^d\times[0,C])$ onto $\cM_2(\R^d)$.
Hence, by Lemma~\ref{lem:mollify}, the mapping
\begin{equation}\label{flow_lifting_defn}
    \tilde{\mu}_s^{t,\mu}:=\cT^*\rho_s^{t,\rho},\quad 0\leq t\leq s\leq T,\,\mu=\cT^*\rho\in\cM_2(\R^d)
\end{equation}
is well-defined, and satisfies the flow property
\begin{equation}\label{flow_lifting}
    \tilde{\mu}_s^{t,\tilde{\mu}_t}=\tilde{\mu}_s,\quad\tilde{\mu}_s^{r,\tilde{\mu}_r^{t,\mu}}=\tilde{\mu}_s^{t,\mu}
\end{equation}
for each $0\leq t\leq r\leq s\leq T$.
It is also clear that $\tilde{\mu}_s^{t,\mu}\in\cM_{2+\delta}(\R^d)$ if $\mu\in\cM_{2+\delta}(\R^d)$.

By \eqref{flow_lifting_defn}, the value function $U$ is defined as
\begin{equation}\label{value_function}
    U(t,\mu):=G(\tilde{\mu}_T^{t,\mu}),\quad (t,\mu)\in[0,T]\times \cM_2(\R^d).
\end{equation}

\begin{prop}\label{prop_bKol}
Suppose that Assumptions \ref{assum:coefficients}, \ref{assum:nondegen} and \ref{assum:differentiability} hold. 
For the value function $U$ defined by \eqref{value_function}, we have:

\emph{(i)} $U$ is $\cC^{1,2}$ on $[0,T]\times(\cM_{2+\delta}(\R^d)\setminus \{0\})$;

\emph{(ii)} $\delta_\mu U,\,D_\mu U,\, D_x^2\delta_\mu U$ are bounded on $[0,T]\times \mathcal{B}_2(C)\times\R^d$ and
$\delta_\mu^2 U,\,\D_\mu^2 U$ are bounded on $[0,T]\times \mathcal{B}_2(C)\times\R^d\times\R^d$, for all $C>0$, where $\mathcal{B}_p(C):=\{\mu\in\cM_p(\R^d):\mu(\R^d)\leq C\},\,p\geq 1$;

\emph{(iii)} For arbitrary $\delta > 0$, $U|_{[0,T]\times\cM_{2+\delta}(\R^d)}$ is a solution to the Kolmogorov backward  equation
\begin{align}\label{eq:bKol}
    &\left\{
    \begin{aligned}
        &\frac{\partial U}{\partial t}+\,\Big\langle \cL \frac{\delta U}{\delta \mu}+\gamma\Big(\sum\limits_l lp_l-1\Big)\frac{\delta U}{\delta \mu} ,\mu\Big\rangle=0,\quad t\in[0,T),\mu\in\cM_{2+\delta}(\R^d),\\
        &U(T,\cdot)=G(\cdot).
    \end{aligned}
    \right.
\end{align}
\end{prop}

Notice that in above, for $\nu\in\cM_{2+\delta}(\R^d)$ and $g:[0,T]\times\R^d\times\cM_{2+\delta}(\R^d)\to\R$, we denote
\[\langle g,\nu\rangle:=\int_{\R^d} g(t,x,\nu)\,\nu(\d x). \]

\subsection{Proof of Proposition~\ref{prop_bKol}}\label{proof_value_regularity}
Firstly, we observe, for each $g\in\cC_b (\R^d\times[0,\infty))$, that (recall \eqref{lifting_inverse})
\[\langle g,\Phi\mu\rangle=
    \begin{cases}
        g(0,0), &\text{if\quad}\mu=0,\\
        \frac{1}{\mu(\R^d)}\big\langle g(\cdot,\mu(\R^d)),\mu\big\rangle, &\text{otherwise.}
    \end{cases}\]
Hence, combining the above with \eqref{function_dual} we can lift or project functionals as follows.

\begin{lem}\label{regularity_G_lifting}
    Let $G\in\cC_b^2(\cM_2(\R^d))$ and $F\in\cC_b^2(\cP_2(\R^d\times[0,C]))$ be differentiable functionals, and $\tilde{G}:=G\circ\cT^*,\,\hat{F}:=F\circ\Phi$ their ``lifted'' and ``projected'' functionals, respectively.
    Suppose there exists a first-order linear derivative $\delta_\rho F:=\delta_\rho F(\rho,(y,z))$ of $F$, which is in direct proportion to $z$, \ie~$\partial_z^2\delta_\rho F=0,\,\delta_\rho F(\rho,(y,0))\equiv0$, and a second-order linear derivative 
    \[\delta_\rho^2 F:=\delta_\rho^2 F(\rho,(y,z),(y^\prime,z^\prime))\]
    in direct proportion to $zz^\prime$. 
    Then it holds that:\\ 
    \emph{(i)} there exist derivatives $\delta_\rho\tilde{G}$ and $\delta_\rho^2\tilde{G}$ of $\tilde{G}$, such that  
    \begin{equation}\label{first_transform}
        \delta_\rho \tilde{G}(\rho,(y,z))=\,\delta_\mu G(\cT^*\rho,y)z
    \end{equation}
    and
    \begin{equation}\label{second_transform}
        \delta_\rho^2 \tilde{G}\big(\rho,(y,z),(y^\prime,z^\prime)\big)=\,\delta_\mu^2 G(\cT^*\rho,y,y^\prime)zz^\prime
    \end{equation}
    for all $\rho\in\cP_2(\R^d\times[0,C]),\,(y,z),(y^\prime,z^\prime)\in\R^d\times[0,C]$;\\
    \emph{(ii)} there exist derivatives $\delta_\mu\hat{F}$ and $\delta_\mu^2\hat{F}$ of $\hat{F}$, such that  
    \begin{equation}\label{first_inverse_transform}
        \delta_\rho F(\Phi\mu,(y,z))=\,\delta_\mu \hat{F}(\mu,y)z
    \end{equation}
    and
    \begin{equation}\label{second_inverse_transform}
        \delta_\rho^2 F\big(\Phi\mu,(y,z),(y^\prime,z^\prime)\big)=\,\delta_\mu^2 \hat{F}(\mu,y,y^\prime)zz^\prime
    \end{equation}
    for all $\mu\in\cM_2(\R^d),\,(y,z),(y^\prime,z^\prime)\in\R^d\times[0,C]$.\\
\end{lem}
\begin{proof}
    \emph{(i)}
    For $G:\cC_b^2(\cM_2(\R^d)),\,\rho,\rho_0\in\cP_2(\R^d\times[0,C])$, it holds, by definition of linear derivative,
    \[
        \begin{aligned}
            \tilde{G}(\rho)-\tilde{G}(\rho_0)=&\,G(\cT^*\rho)-G(\cT^*\rho_0)\\
            =&\,\int_0^1\Big\langle\delta_\mu G\big(\lambda\cT^*\rho+(1-\lambda)\cT^*\rho_0,\cdot\big),\cT^*\rho-\cT^*\rho_0\Big\rangle\,\d\lambda\\
            =&\,\int_0^1\Big\langle\cT\delta_\mu G\big(\cT^*(\lambda\rho+(1-\lambda)\rho_0)\big)(\cdot),\rho-\rho_0\Big\rangle\,\d\lambda.
        \end{aligned}
    \]
    Hence, we have, for all $\rho\in\cP_2(\R^d\times[0,C])$,
    \[\frac{\delta \tilde{G}}{\delta\rho}(\rho,\cdot)=\cT\frac{\delta G}{\delta\mu}(\cT^*\rho,\cdot),\]
    which yields \eqref{first_transform}.
    Analogously, \eqref{second_transform} holds, and the proof is complete.\\
    \emph{(ii)}
    For $F:\cC_b^2(\cP_2(\R^d\times[0,C])),\,\mu,\mu_0\in\cM_2(\R^d)$, it holds, by definition of linear derivative,
    \[
        \begin{aligned}
            \hat{F}(\mu)-\hat{F}(\mu_0)=&\,F(\Phi\mu)-F(\Phi(\mu_0))\\
            =&\,\int_0^1\Big\langle\delta_\rho F\big(\lambda\Phi\mu+(1-\lambda)\Phi(\mu_0),\cdot\big),\Phi\mu-\Phi(\mu_0)\Big\rangle\,\d\lambda.
        \end{aligned}
    \]
    Note that
    \[
    \begin{aligned}
        &\, \Big\langle\delta_\rho F\big(\lambda\Phi\mu+(1-\lambda)\Phi(\mu_0),\cdot\big),\Phi\mu\Big\rangle\\
        =\,&
    \begin{cases}
        \delta_\rho F\big(\lambda\Phi\mu+(1-\lambda)\Phi(\mu_0),0,0\big), &\text{if\quad}\mu=0,\\
        \frac{1}{\mu(\R^d)}\Big\langle \delta_\rho F\big(\lambda\Phi\mu+(1-\lambda)\Phi(\mu_0),\cdot,\mu(R^d)\big),\mu\Big\rangle, &\text{otherwise.}
    \end{cases}\\
    =\,&\Big\langle\partial_z\delta_\rho F\big(\lambda\Phi\mu+(1-\lambda)\Phi(\mu_0)\big)(\cdot),\mu\Big\rangle,
    \end{aligned}\]
    and similarly that
    \[
        \Big\langle\delta_\rho F\big(\lambda\Phi\mu+(1-\lambda)\Phi(\mu_0),\cdot\big),\Phi(\mu_0)\Big\rangle=\,\Big\langle\partial_z\delta_\rho F\big(\lambda\Phi\mu+(1-\lambda)\Phi(\mu_0)\big)(\cdot),\mu_0\Big\rangle,\]
    Hence, we have
    \[
        \hat{F}(\mu)-\hat{F}(\mu_0)=\,\int_0^1\Big\langle\partial_z\delta_\rho F\big(\lambda\Phi\mu+(1-\lambda)\Phi(\mu_0)\big)(\cdot),\mu-\mu_0\Big\rangle\,\d\lambda.
    \]
    which yields $\delta_\mu\hat{F}(\mu,\cdot)=\partial_z\delta_\rho F(\Phi\mu,\cdot)$ and thus \eqref{first_inverse_transform}.
    Analogously, \eqref{second_inverse_transform} holds, and the proof is complete.
    \qed
\end{proof}

We also show that $\Phi$ is continuous on $\cM_{2+\delta}(\R^d)\setminus\{0\}$ for all $\delta>0$.
\begin{lem}\label{lifting_inverse_continuity}
    Let $\{\nu_k\}_{k\geq1}\subset \cM_2(\R^d)$ be a sequence converging to some $\nu\in\cM_2(\R^d)\setminus\{0\}$ in $\mathbf{d}$.
    If $\limsup_{k\to\infty}\nu_k(\abs{\cdot}^{2+\delta})<\infty$ for some $\delta>0$, then it holds that
    $\Phi\nu_k$ converges to $\Phi\nu$ in $\cW_2$.
\end{lem}
\begin{proof}
    By definition of Wasserstein distances, it holds that
    \[
    \begin{aligned}
        \cW_2(\Phi\nu_k,\Phi\nu)\leq\,& \sqrt{\cW_2(\bar{\nu_k},\bar{\nu})^2+\cW_2(\delta_{\nu_k(\R^d)},\delta_{\nu(\R^d)})^2}\\
        \leq\,&\cW_2(\bar{\nu_k},\bar{\nu})+\abs{\nu_k(\R^d)-\nu(\R^d)}.
    \end{aligned}\]
    Hence, it suffices to prove $\cW_2(\bar{\nu_k},\bar{\nu})\to0$ and $\abs{\nu_k(\R^d)-\nu(\R^d)}\to 0$.
    By definition of $\mathbf{d}$, we have 
    \[\abs{\nu_k(\R^d)-\nu(\R^d)}=\abs{\langle1,\nu_k-\nu\rangle}\leq \mathbf{d}(\nu_k,\nu)\to0. \]
    Thus their exists some $N_0>0$ such that $0<\nu(\R^d)/2\leq\nu_k\leq 3\nu(\R^d)/2$ for all $k\geq N_0$.
    For each $k\geq N_0$ and Lipschitz function $f$ on $\R^d$ such that $\norm{f}_{\infty},\norm{f}_{\text{Lip}}\leq 1$, we have
    \[
        \langle f,\bar{\nu_k}-\bar{\nu}\rangle=\, \frac{1}{\nu_k(\R^d)}\langle f,\nu_k-\nu\rangle+\langle f,\nu\rangle(\frac{1}{\nu_k(\R^d)}-\frac{1}{\nu(\R^d)}),\]
    and thus 
     \[
    \begin{aligned}
        \mathbf{d}(\bar{\nu_k},\bar{\nu})=\,&\sup_{\norm{f}_{\infty},\norm{f}_{\text{Lip}}\leq 1}\langle f,\bar{\nu_k}-\bar{\nu}\rangle\\
        \leq\,&\frac{1}{\nu_k(\R^d)}\mathbf{d}(\nu_k,\nu)+\nu(\R^d)\abs{\frac{1}{\nu_k(\R^d)}-\frac{1}{\nu(\R^d)}}\to0.
    \end{aligned}\]
    We observe that $\mathbf{d}$ coincides with the $1$-Wasserstein distance with respect to the metric $\abs{\cdot}\wedge 2$, which induces the same topology as $\abs{\cdot}$.
    Hence, for sequences in $\cP_2(\R^d)$, convergence in $\mathbf{d}$ yields the weak convergence, and by \cite[Definition 6.8 and Theorem 6.9]{villani2008optimal}, it suffices to prove
    \[\lim_{k\to\infty}\bar{\nu_k}(\abs{\cdot}^2)=\bar{\nu}(\abs{\cdot}^2).\]
    Since $\limsup_{k\to\infty}\nu_k(\abs{\cdot}^{2+\delta}) <\infty,\,\lim_{k\to\infty}\nu_k(\R^d)=\nu(\R^d)>0$, we have
    \[\limsup_{k\to\infty}\bar{\nu_k}(\abs{\cdot}^{2+\delta})<\infty.\]
    Without loss of generality, we assume
    \[\sup_{k\geq1}\bar{\nu_k}(\abs{\cdot}^{2+\delta})\leq C\]
    for some $C\in[0,\infty)$, and thus $\bar{\nu}(\abs{\cdot}^{2+\delta})\leq C$.
    For each $M\geq1$, we have. by triangle inequality and Chebyshev's inequality,
    \[
    \begin{aligned}
        \abs{\bar{\nu_k}(\abs{\cdot}^2)-\bar{\nu}(\abs{\cdot}^2)}\leq\,& \abs{\bar{\nu_k}(\abs{\cdot}^2)-\bar{\nu_k}(\abs{\cdot}^2\wedge M)}+\abs{\bar{\nu_k}(\abs{\cdot}^2\wedge M)-\bar{\nu}(\abs{\cdot}^2\wedge M)}\\
        &\,+\abs{\bar{\nu}(\abs{\cdot}^2\wedge M)-\bar{\nu}(\abs{\cdot}^2)}\\
        =\,& \abs{\bar{\nu_k}(\abs{\cdot}^2\wedge M)-\bar{\nu}(\abs{\cdot}^2\wedge M)}\\
        &\,+\bar{\nu_k}(\abs{\cdot}^2\mathbbm{1}_{\abs{\cdot}>M})+\bar{\nu}(\abs{\cdot}^2\mathbbm{1}_{\abs{\cdot}>M})\\
        \leq\,&\abs{\bar{\nu_k}(\abs{\cdot}^2\wedge M)-\bar{\nu}(\abs{\cdot}^2\wedge M)}\\
        &\,+M^{-\delta}\big(\bar{\nu_k}(\abs{\cdot}^{2+\delta})+\bar{\nu}(\abs{\cdot}^{2+\delta})\big)\\
        \leq\,&\abs{\bar{\nu_k}(\abs{\cdot}^2\wedge M)-\bar{\nu}(\abs{\cdot}^2\wedge M)}+2CM^{-\delta}.
    \end{aligned}\]
    Note that $\abs{\cdot}^2\wedge M$ is bounded Lipschitz, and thus we let $k\to\infty$ and obtain
    \[\limsup_{k\to\infty}\abs{\bar{\nu_k}(\abs{\cdot}^2)-\bar{\nu}(\abs{\cdot}^2)}\leq 2CM^{-\delta}.\]
    Let $M\to\infty$, and we have $\bar{\nu_k}(\abs{\cdot}^2)\to \bar{\nu}(\abs{\cdot}^2)$.
\qed
\end{proof}

Now we turn back to regularity of $U$.
Let $\tilde{U}(t,\rho):=\tilde{G}(\rho_T^{t,\rho})=G(\cT^*\rho_T^{t,\rho})$, and by \eqref{flow_lifting_defn} and \eqref{value_function} we have, for each $\mu\in\cM_2(\R^d),\,\rho\in\cP_2(\R^d\times [0,C])$,
    \[\tilde{U}(t,\Phi\mu)=U(t,\mu),\quad \tilde{U}(t,\rho)=U(t,\cT^*\rho),\quad t\in[0,T].\]
Note that regularity of $\tilde{G}$ holds by Lemma~\ref{regularity_G_lifting}, and thus intrinsic derivatives $D_\rho\tilde{U}, D_\rho^2\tilde{U}$ exist (cf.~\cite[Lemmas 6.1, 6.2]{10.1214/15-AOP1076}).
The following argument addresses the regularity of intrinsic derivatives, which consequently yields the existence of linear derivatives $\delta_\mu U,\delta_\mu^2 U$.

\begin{lem}\label{intrinsic_is_linear}
    Suppose Assumptions \ref{assum:coefficients}, \ref{assum:nondegen} and \ref{assum:differentiability} hold.
    Then for all $t\in[0,T]$, $\rho\in\cP_2(\R^d\times[0,C])$, it holds that:\\
    \emph{(i)} for $1\leq i\leq d$, the $i$th entry of the intrinsic derivative $D_\rho \tilde{U}$ is in direct proportion to the $(d+1)$th variable $z$, \ie~it assumes the form $(D_\rho\tilde{U})_i(t,\rho,(y,z))=f_i(t,\rho,y)z$;\\
    \emph{(ii)} the $(d+1)$th entry of $D_\rho \tilde{U}$ is independent of $z$, \ie~$\partial_z(D_\rho\tilde{U})_{d+1}(t,\rho,(y,z))\equiv0$.
\end{lem}
The proof of Lemma~\ref{intrinsic_is_linear} is given in Section~\ref{proof_intrinsic_is_linear}, and the second-order version can be obtained analogously.
The above analysis can directly lead to the following proposition on linear derivatives.
\begin{prop}\label{linear_is_linear}
    Suppose Assumptions~\ref{assum:coefficients}, \ref{assum:nondegen} and \ref{assum:differentiability} hold.
    Then for all $t\in[0,T]$, $\rho\in\cP_2(\R^d\times[0,C])$, it holds that:\\
    \emph{(i)} there exists a linear derivative $\delta_\rho \tilde{U}=\delta_\rho \tilde{U}(t,\rho,(y,z))$ in direct proportion to $z$; \\
    \emph{(ii)} there exists a linear derivative $\delta^2_\rho \tilde{U}=\delta^2_\rho \tilde{U}(t,\rho,(y,z),(y^\prime,z^\prime))$ in direct proportion to $zz^\prime$.
\end{prop}
\begin{proof}
    For simplicity, we only prove (i) here, and the proof of (ii) is similar.
    By \cite[Propositions 5.50, 5.51]{carmona2018probabilistic}, the term
    \[
    \begin{aligned}
        p^{t,\rho}(y,z):=\,&\int_0^1 D_\rho\tilde{U}(t,\rho,(sy,sz))\cdot(y,z)\,\d s\\
        =\,&\int_0^1\Big(\sum_{i=1}^d (D_\rho\tilde{U})_{i}(t,\rho,(sy,sz))y_i+(D_\rho\tilde{U})_{d+1}(t,\rho,(sy,sz))z\Big)\,\d s
    \end{aligned}\]
    is a linear derivative $\delta_\rho \tilde{U}$ of $\tilde{U}$.
    Hence, it holds, by Lemma~\ref{intrinsic_is_linear}, that
    \[
    \begin{aligned}
        \delta_\rho \tilde{U}(t,\rho,(y,z))=\,&\int_0^1\Big(\sum_{i=1}^d f_{i}(t,\rho,sy)szy_i+(D_\rho\tilde{U})_{d+1}(t,\rho,sy)z\Big)\,\d s\\
        =\,&z\int_0^1\Big(\sum_{i=1}^d f_{i}(t,\rho,sy)sy_i+(D_\rho\tilde{U})_{d+1}(t,\rho,sy)\Big)\,\d s,
    \end{aligned}\]
    which completes the proof.
    \qed
\end{proof}

Combining Lemma~\ref{regularity_G_lifting} and Proposition~\ref{linear_is_linear}, we obtain, for each fixed $t\in[0,T]$, the following result.
\begin{prop}
    Suppose Assumptions~\ref{assum:coefficients}, \ref{assum:nondegen} and \ref{assum:differentiability} hold.
    For $t\in[0,T]$, $\mu\in\cM_2(\R^d)$ and $\rho=\Phi\mu$, it holds
    \[\frac{\delta\tilde{U}}{\delta\rho}\big(t,\rho,(y,z)\big)=\frac{\delta U}{\delta\mu}(t,\mu,y)z\]
    for all $(y,z)\in\R^d\times[0,C]$, and
    \[\frac{\delta^2\tilde{U}}{\delta\rho^2}\big(t,\rho,(y,z),(y^\prime,z^\prime)\big)=\,\frac{\delta^2 U}{\delta\mu^2}(t,\mu,y,y^\prime)zz^\prime\]
    for all $(y,z),(y^\prime,z^\prime)\in\R^d\times[0,C]$.
\end{prop}
Hence, we have, for each $(t,\mu,y,y^\prime)\in[0,T]\times\cM_2(\R^d)\times\R^d\times\R^d$,
\begin{align*}
    \frac{\delta U}{\delta\mu}(t,\mu,y)=\,&(D_\rho)_{d+1}\tilde{U}\big(t,\rho,(y,z)\big),\\
    (D_\mu)_i U(t,\mu,y)=\,&(D_\rho)_i\tilde{U}\big(t,\rho,(y,1)\big),\,1\leq i\leq d,\\
    \frac{\delta^2 U}{\delta\mu^2}(t,\mu,y,y^\prime)=\,&(D_\rho^2)_{d+1,d+1}\tilde{U}\big(t,\rho,(y,1),(y^\prime,1)\big),\\(D_\mu^2)_{i,j} U(t,\mu,y,y^\prime)=\,&(D_\rho)_{i,j}\tilde{U}\big(t,\rho,(y,1),(y^\prime,1)\big),\,1\leq i,j\leq d,
\end{align*}
where $\rho=\Phi\mu$.
On the other hand, by argument similar to \cite[(7.9), (7.10)]{10.1214/15-AOP1076}, we have
\[
\begin{aligned}
    \partial_t U(t,\mu)=-\Big\langle \Big(\cL \delta_\mu G+\gamma\big(\sum\limits_l lp_l-1\big)\delta_\mu G\Big)(t,\cdot,\tilde{\mu}_T^{t,\mu}),\tilde{\mu}_T^{t,\mu}\Big\rangle.
\end{aligned}
\]
Combining the above with \cite[Lemmas 6.1, 6.2]{10.1214/15-AOP1076}, we can prove, by Assumptions~\ref{assum:coefficients} and \ref{assum:differentiability}, (i) and (ii) of Proposition~\ref{prop_bKol}.

\begin{rem}
    In \cite{10.1214/15-AOP1076}, the mean-field dynamics is time homogeneous, while in our work the coefficients can rely on time.
    The time homogeneous nature, however, was not used in \cite{10.1214/15-AOP1076} except for regularity of the value function with respect to $t$, which can be handled by replacing $x$ with $\tilde{x}=(t,x)$.
    On the other hand, the Lipschitz properties in our work are all uniform in time, and we do not need to compute the derivatives of $b,\sigma,c$ with respect to $t$.
    Hence, regularity of $\tilde{U}$ and consequently $U$ can be obtained analogously to the time homogeneous scenarios.    
\end{rem}

\begin{rem}
    Note that $\tilde{U}(t,\cdot)$ is defined on $\cP_2(\R^d\times [0,C])$ instead of $\cP_2(\R^d)$ or $\cP_2(\R^{d+1})$, which is different from \cite{10.1214/15-AOP1076,carmona2018probabilistic} scenarios.
    The regularity of $\tilde{U}$, however, can be obtained with the same method.
\end{rem}
    
\begin{rem}
	One can in fact prove the existence and boundedness for derivatives of $U$ on $\cM_2(\R^d)$.
	However, for the continuity of the derivatives, we require that $\mu\in\cM_{2+\delta}(\R^d)\setminus\{0\}$ to have some additional tightness property (refer to Lemma \ref{lifting_inverse_continuity}).
\end{rem}

It remains to verify that $U|_{[0,T]\times\cM_{2+\delta}(\R^d)}$ is a solution to \eqref{eq:bKol}.
Note that $U(\cdot,0)\equiv G(0)$ and thus it suffices to prove for $\mu\neq 0$ scenarios,
where $\tilde{\mu}_T^{t,\mu}\neq 0$ for each $t\in [0,T]$.
For all $0\leq t\leq s\leq T,\,\mu\in\cM_{2+\delta}(\R^d)$, we have, by \eqref{flow_lifting}, 
\[U(s,\tilde{\mu}_s^{t,\mu})=G\big(\tilde{\mu}_T^{s,\tilde{\mu}_s^{t,\mu}}\big)=G(\tilde{\mu}_T^{t,\mu})=U(t,\mu).\]
Hence, applying It\^o's formula (see Theorem \ref{ito_formula}) we have
\[
\begin{aligned}
    0=\,&U(s,\tilde{\mu}_s^{t,\mu})-U(t,\mu)\\
    =\,&\int_t^s\bigg(\partial_r U(r,\tilde{\mu}_r^{t,\mu})+\Big\langle \cL \frac{\delta U}{\delta \mu}+\gamma\Big(\sum\limits_l lp_l-1\Big)\frac{\delta U}{\delta \mu} ,\tilde{\mu}_r^{t,\mu}\Big\rangle\bigg)\,\d t,
\end{aligned}
\]
and thus, by continuity,
\[\partial_t U(t,\mu)+\Big\langle \cL \frac{\delta U}{\delta \mu}+\gamma\Big(\sum\limits_l lp_l-1\Big)\frac{\delta U}{\delta \mu} ,\mu\Big\rangle=0.\]
On the other hand, it is clear that $U(T,\cdot)=G(\tilde{\mu}_T^{T,\cdot})=G(\cdot)$.

Combining the above, we complete the proof of Proposition~\ref{prop_bKol}.

\subsection{Proof of Theorem~\ref{thm_WPoC}}\label{proof_main}

By Lemma \ref{lem:FKP} together with the uniqueness result of the nonlinear Fokker--Planck equation in Proposition~\ref{prop:uniqueness},
it follows that, for $\tilde \mu_0 = \mu_0\in\cM_{2+\delta}^\theta(\R^d)$,
\[\mu_T=\tilde{\mu}_T=\tilde{\mu}_T^{0,\mu_0}.\]
Hence, applying Proposition~\ref{prop_bKol} and Theorem~\ref{ito_formula}, one obtains that
\begin{equation}\label{functional_taylor}
    \begin{aligned}
        &\quad G(\mu_T)-\Expect[G(\mu_T^N)]=\,U(0,\mu_0)-\Expect[U(0,\mu_0^N)]-\Expect[U(T,\mu_T^N)-U(0,\mu_0^N)]\\
        =\,& U(0,\mu_0)-\Expect[U(0,\mu_0^N)]
        \\
        &\, - \frac{1}{2N}\int_0^T \Expect\Big[\int_{\R^d}\trace\big(D_{\mu}^2 U(t,\mu_t^N,x,x)\sigma\sigma^\transpose\big)\mu_t^N(\d x)\Big]\d t\\
        &\, +\int_0^T \Expect\Big[\Big\langle\gamma\Big(\sum\limits_l lp_l-1\Big)\frac{\delta U}{\delta \mu} ,\mu_t^N\Big\rangle\Big]\d t\\
        &\, -N\int_0^T \Expect\Big[\int_{\R^d}\gamma\sum\limits_{l}p_l\Big(U\big(t,\mu_t^N+\frac{l-1}{N}\delta_x\big)-U\big(t,\mu_t^N\big)\Big)\mu_t^N(\d x)\Big]\d t,
    \end{aligned}
\end{equation}
where we use 
\[\partial_t U(t,\mu_t^N)+\Big\langle \cL \frac{\delta U}{\delta \mu}+\gamma\Big(\sum\limits_l lp_l-1\Big)\frac{\delta U}{\delta \mu} ,\mu_t^N\Big\rangle=0\]
for all $t\in [0,T)$.
Note that
\[
\begin{aligned}
    &\quad U\big(t,\mu_t^N+\frac{l-1}{N}\delta_x\big)-U\big(t,\mu_t^N\big)=\,\int_0^1 \frac{l-1}{N}\frac{\delta U}{\delta\mu}\big(t,\mu_t^N+\lambda\frac{l-1}{N}\delta_x,x\big)\,\d\lambda\\
    =\,&\frac{l-1}{N}\bigg(\frac{\delta U}{\delta\mu}(t,\mu_t^N,x)+\int_0^1 \Big(\frac{\delta U}{\delta\mu}\big(t,\mu_t^N+\lambda\frac{l-1}{N}\delta_x,x\big)-\frac{\delta U}{\delta\mu}(t,\mu_t^N,x)\Big)\,\d\lambda\bigg)\\
    =\,&\frac{l-1}{N}\frac{\delta U}{\delta\mu}(t,\mu_t^N,x)\\
    &\,+\frac{l-1}{N}\int_0^1\d\lambda\bigg(\lambda\frac{l-1}{N}\int_0^1\frac{\delta^2 U}{\delta\mu^2}\big(t,\mu_t^N+\lambda\lambda^\prime\frac{l-1}{N}\delta_x,x,x\big)\,\d\lambda^\prime\bigg)\\
    =\,&\frac{l-1}{N}\frac{\delta U}{\delta\mu}(t,\mu_t^N,x)+\frac{(l-1)^2}{2N^2}\frac{\delta^2 U}{\delta\mu^2}\big(t,\mu_t^N+\eta\frac{l-1}{N}\delta_x,x,x\big),
\end{aligned}\]
where we apply the mean value theorem for integrals in the fifth line, and $\eta\in[0,1]$.
Hence, we have
\[
\begin{aligned}
   &\quad G(\mu_T)-\Expect[G(\mu_T^N)]\\
        =\,& U(0,\mu_0)-\Expect[U(0,\mu_0^N)]
        \\
        &\, - \frac{1}{2N}\int_0^T \Expect\Big[\int_{\R^d}\trace\big(D_{\mu}^2 U(t,\mu_t^N,x,x)\sigma\sigma^\transpose\big)\mu_t^N(\d x)\Big]\d t\\
        &\, -\frac{1}{2N}\int_0^T \Expect\Big[\int_{\R^d}\gamma\sum_l p_l(l-1)^2\frac{\delta^2 U}{\delta\mu^2}\big(t,\mu_t^N+\eta(l)\frac{l-1}{N}\delta_x,x,x\big)\,\mu_t^N(\d x)\Big]\d t, 
\end{aligned}
\]
and thus
\begin{equation}\label{final_inequality}
    \begin{aligned}
   &\quad \abs{G(\mu_T)-\Expect[G(\mu_T^N)]}\\
        \leq\,& \abs{U(0,\mu_0)-\Expect[U(0,\mu_0^N)]}
        \\
        &\, +\frac{1}{2N}\int_0^T \Expect\Big[\int_{\R^d}\sup\abs{\trace\big(D_{\mu}^2 U\sigma\sigma^\transpose\big)}\mu_t^N(\d x)\Big]\d t\\
        &\, +\frac{1}{2N}\int_0^T \Expect\Big[\int_{\R^d}\gamma\sum_l p_l(l-1)^2\sup\abs{\frac{\delta^2 U}{\delta\mu^2}}\,\mu_t^N(\d x)\Big]\d t.
\end{aligned}
\end{equation}
	By the boundedness result of the derivatives in Proposition \ref{prop_bKol}, together with Lemma \ref{lem_initial} below, 
	we then conclude the proof of Theorem \ref{thm_WPoC}.
\qed

\begin{lem}\label{lem_initial}
    Under Assumptions~\ref{assum:coefficients}, \ref{assum:initial}, \ref{assum:nondegen} and \ref{assum:differentiability}, we have
    \[\abs{U(0,\mu_0)-\Expect[U(0,\mu_0^N)]}\leq C\frac{1}{N}.\]
\end{lem}
\begin{proof}
	We follow the technique in \cite[Theorem 2.14]{chassagneux2022weak}.
    For simplicity, we write $U_0(\cdot):=U(0,\cdot),\,\mu_0^{N,\lambda}:=\lambda\mu_0^N+(1-\lambda)\mu_0$ for $\lambda\in[0,1]$.
    Note that the regularity of $U_0$ follows from Proposition~\ref{prop_bKol},
    and we have, by Assumption~\ref{assum:initial} and definition of linear derivatives,
    \[
    \begin{aligned}
        &\quad\Expect[U_0(\mu_0^N)-U_0(\mu_0)]=\,\Expect\Big[\int_0^1\langle\delta_\mu U_0(\mu_0^{N,\lambda},\cdot),\mu_0^N-\mu_0\rangle\,\d\lambda\Big]\\
        =\,&\int_0^1\frac{1}{N}\sum_{i=1}^N\Expect\Big[\sum_{k\in K_0^i}\delta_\mu U_0(\mu_0^{N,\lambda},X_0^{i,k})-\sum_{k\in K_0}\delta_\mu U_0(\mu_0^{N,\lambda},X_0^k)\Big]\,\d\lambda\\
        =\,&\int_0^1\Expect\Big[\sum_{k\in K_0^1}\delta_\mu U_0(\mu_0^{N,\lambda},X_0^{1,k})-\sum_{k\in K_0}\delta_\mu U_0(\mu_0^{N,\lambda},X_0^k)\Big]\,\d\lambda,
    \end{aligned}\]
    where we use the i.i.d. property of $\{Z_0^i\}_{i=1}^N$ in the last step.
    Without loss of generality, we assume $\bar{Z}_0$ is independent of $\{Z_0^i\}_{i=1}^N$,
    and write 
    \[
    \begin{aligned}
        \bar{\mu}_0^{N,\lambda}:=\,&\mu_0^{N,\lambda}+\frac{\lambda}{N}\Big(\sum_{k\in K_0}\delta_{X_0^k}-\sum_{k\in K_0^1}\delta_{X_0^{1,k}}\Big)\\
        \bar{\mu}_0^{N,\lambda,\lambda^\prime}:=\,&\lambda^\prime \bar{\mu}_0^{N,\lambda}+(1-\lambda^\prime)\mu_0^{N,\lambda}
    \end{aligned}\]
    for $\lambda,\lambda^\prime\in[0,1]$.
    Utilizing independence and identical distribution, we have
    \[\Expect\Big[\sum_{k\in K_0^1}\delta_\mu U_0(\mu_0^{N,\lambda},X_0^{1,k})\Big]=\Expect\Big[\sum_{k\in K_0}\delta_\mu U_0(\bar{\mu}_0^{N,\lambda},X_0^k)\Big],\]
    and thus
    \[
    \begin{aligned}
        &\quad\Expect[U_0(\mu_0^N)-U_0(\mu_0)]=\,\int_0^1\Expect\Big[\sum_{k\in K_0}\big(\delta_\mu U_0(\bar{\mu}_0^{N,\lambda},X_0^k)-\delta_\mu U_0(\mu_0^{N,\lambda},X_0^k)\big)\Big]\,\d\lambda\\
        =\,&\int_0^1\d\lambda\frac{\lambda}{N}\Expect\Big[\sum_{k\in K_0}\int_0^1\Big(\sum_{k\in K_0}\delta_\mu^2 U_0(\bar{\mu}_0^{N,\lambda,\lambda^\prime},X_0^k,X_0^k)\\
        &\,-\sum_{k\in K_0^1}\delta_\mu^2 U_0(\bar{\mu}_0^{N,\lambda,\lambda^\prime},X_0^k,X_0^{1,k})\Big)\,\d\lambda^\prime\Big].
    \end{aligned}\]
    By Proposition~\ref{prop_bKol}, we have
    \[
    \begin{aligned}
        \abs{\Expect[U_0(\mu_0^N)-U_0(\mu_0)]}\leq\,&\int_0^1\d\lambda\frac{\lambda}{N}\Expect\Big[\sum_{k\in K_0}(\# K_0+\#K_0^1)\sup\abs{\delta_\mu^2 U}\Big]\\
        \leq\,&\frac{C}{2N}\Big(\Expect\big[(\# K_0)^2\big]+\Expect\big[(\# K_0)^2\big]^{1/2}\Expect\big[(\# K_0^1)^2\big]^{1/2}\Big)\\
        \leq\,& C\frac{1}{N},
    \end{aligned}\]
    which completes the proof.
    \qed
\end{proof}

\begin{appendix}

\section{Proof of Theorem~\ref{ito_formula}}\label{append_ito}
In this section, we prove functional It\^o's formula \eqref{ito_environment} for the environment measure $\mu_t$ and \eqref{ito_empirical} for the empirical measure $\mu_t^N$, as detailed in the following two subsections, respectively.
\subsection{It\^o's formula for $\mu_t$}
Firstly, we prove the continuity of $\mu_t$ with respect to $t$.
\begin{lem}\label{time_continuity}
    For all $0\leq t\leq s\le T$, it holds 
    \[\mathbf{d}(\mu_s,\mu_t)\leq C\abs{s-t}^{1/2}\]
    for some constant $C=C(d,T,M,\bar{\gamma},\Expect[\# K_0])$.
\end{lem}
\begin{proof}
    For each bounded Lipschitz function $f:\R^d\to\R$ such that $\norm{f}_{\infty},\norm{f}_{\text{Lip}}\leq 1$, we have
    \begin{align*}
        \abs{\langle f,\mu_s-\mu_t\rangle}\leq\,& \Expect\Big[\sum_{k\in K_s\cap K_t}\abs{f(x_s^k)-f(X_t^k)}\Big]\\
        &\,+\Big\vert\Expect\Big[\sum_{k\in K_s\setminus K_t}f(x_s^k)-\sum_{k\in K_t\setminus K_s}f(x_t^k)\Big]\Big\vert.
    \end{align*}
    Note that
    \begin{align*}
        \Expect\Big[\sum_{k\in K_s\cap K_t}\abs{f(x_s^k)-f(X_t^k)}\Big]\leq\,& \Expect\Big[\sum_{k\in \mathbb{K}}\abs{X_s^k-X_t^k}\mathbbm{1}_{K_s\cap K_t}(k)\Big]\\
        =\,&\sum_{k\in \mathbb{K}}\Expect\Big[\abs{X_s^k-X_t^k}\Big\vert k\in K_s\cap K_t\Big]\Expect[1_{\{k\in K_s\cap K_t\}}]\\
        \leq\,& \sup_{k\in\mathbb{K}} \Expect \Big[\abs{X_s^k-X_t^k}\Big]\Expect\Big[\sum_{k\in\mathbb{K}}1_{\{k\in K_s\cap K_t\}}\Big].
    \end{align*}
    By Assumption~\ref{assum:coefficients}, it holds that
    \[
    \begin{aligned}
        \Expect\Big[\abs{X_s^k-X_t^k}\Big]\leq\,&\Expect\Big[\abs{X_s^k-X_t^k}^2\Big]^{1/2}\\
        \leq\,&\Expect\bigg[\abs{\int_t^s b(r,X_r^k,\mu_r)\,\d r+\int_t^s \sigma(r,X_r^k,\mu_r)\,\d W_r^k}^2\bigg]^{1/2}\\
        \leq\,& M\abs{s-t}+Md\abs{s-t}^{1/2}\leq C\abs{s-t}^{1/2}.
    \end{aligned}\]
    Applying Assumption~\ref{assum:initial}, we have (recall \eqref{mass_growth})
    \[\Expect\Big[\sum_{k\in\mathbb{K}}1_{\{k\in K_s\cap K_t\}}\Big]=\Expect[\#(K_s\cap K_t)]\leq \Expect[\# K_0]\exp(\bar{\gamma} MT)\leq C.\]
    On the other hand, note that
    \[
    \begin{aligned}
        &\quad\Expect\Big[\sum_{k\in K_s\setminus K_t}f(x_s^k)-\sum_{k\in K_t\setminus K_s}f(x_t^k)\Big]\\
        =\,&\Expect\Big[\int_t^s\int_{[0,\bar{\gamma}]\times[0,1]}\sum\limits_{k\in K_{r-}}\sum\limits_{l\geq0}(l-1)f(X_r^k)\mathbbm{1}_{([0,\gamma]\times I_l)(r,X_r^k,\mu_{r-})}\,Q^k(\d r,\d z)\Big]\\
        =\,&\int_t^s\Big\langle \gamma\Big(\sum\limits_l lp_l-1\Big)f(r,\cdot,\mu_r),\mu_r\Big\rangle\,\d r.
    \end{aligned}\]
    Hence, we have
    \[\Big\vert\Expect\Big[\sum_{k\in K_s\setminus K_t}f(X_s^k)-\sum_{k\in K_t\setminus K_s}f(X_t^k)\Big]\Big\vert\leq \int_t^s\Big\langle \bar{\gamma}M,\mu_r\Big\rangle\,\d r\leq C\abs{s-t}\leq C\abs{s-t}^{1/2}.\]

    Combining the above with Definition~\ref{wasserstein}, we complete proof of the lemma.
    \qed
\end{proof}

On the other hand, for each $f\in \cC_b^2(\R^d)$ it holds that
\begin{equation}\label{environment_integral}
    \begin{aligned}
        \langle f,\mu_t\rangle=&\,\langle f,\mu_0\rangle+\int_0^t\big\langle \cL f(s,\cdot,\mu_s),\mu_s\big\rangle\,\d s\\
        &\,+\Expect\Big[\int_0^t\int_{[0,\bar{\gamma}]\times[0,1]}\sum\limits_{k\in K_{s-}}\sum\limits_{l\geq0}(l-1)f(X_s^k)\mathbbm{1}_{([0,\gamma]\times I_l)(s,X_s^k,\mu_{s-})}\,Q^k(\d s,\d z)\Big]\\
        =&\,\langle f,\mu_0\rangle+\int_0^t\Big\langle \Big(\cL f+\gamma\Big(\sum\limits_l lp_l-1\Big)f\Big)(s,\cdot,\mu_s),\mu_s\Big\rangle\,\d s,
    \end{aligned}
\end{equation}
where we use $K_{t-}=K_t,\,\mu_{t-}=\mu_t$ for almost all $t\in[0,T]$.

\paragraph{Proof of the equation \eqref{ito_environment}.}

\emph{Step 1}: Let us begin with $F\in\cC^2(\cM_2(\R^d))$.
For all $t\in[0,T],\,n\geq1$, let $0=t_0<t_1<...<t_n=t$ be a partition of $[0,t]$ such that $t_i-t_{i-1}=\frac{t}{n}$ for $1\leq i\leq n$, and we have
\[F(\mu_t)-F(\mu_0)=\sum\limits_{i=1}^n\big(F(\mu_{t_i})-F(\mu_{t_{i-1}})\big).\]
For $1\leq i\leq n$, it holds
\[
\begin{aligned}
    &\quad F(\mu_{t_i})-F(\mu_{t_{i-1}})\\
    =&\,\int_0^1\big\langle\frac{\delta F}{\delta\mu}\big(\lambda\mu_{t_i}+(1-\lambda)\mu_{t_{i-1}},\cdot\big),\,\mu_{t_i}-\mu_{t_{i-1}}\big\rangle\,\d\lambda\\
    =&\,\int_{t_{i-1}}^{t_i}\d r
    \int_0^1\bigg\langle\Big(\cL+\gamma\big(\sum\limits_l lp_l-1\big)\Big)\frac{\delta F}{\delta\mu}\big(\lambda\mu_{t_i}+(1-\lambda)\mu_{t_{i-1}}\big)\big(r,\cdot,\mu_r\big),\,\mu_r\bigg\rangle\,\d\lambda,
\end{aligned}\]
where we use \eqref{environment_integral} in the last line.
For each $r\in[t_{i-1},t_i)$, denote by $t^n(r):=t_{i-1}$ and by $\bar{t}^n(r):=t_i$, and the above can be rewritten as
\[
    \begin{aligned}
        &\quad F(\mu_t)-F(\mu_0)\\
        =&\,\int_0^t\int_0^1\bigg\langle\Big(\cL+\gamma\big(\sum\limits_l lp_l-1\big)\Big)\frac{\delta F}{\delta\mu}\big(\lambda\mu_{\bar{t}^n(r)}+(1-\lambda)\mu_{t^n(r)}\big)\big(r,\cdot,\mu_r\big),\,\mu_r\bigg\rangle\,\d\lambda\d r.
    \end{aligned}
\]
We write 
\[g_n(r,\lambda):=\bigg\langle\Big(\cL+\gamma\big(\sum\limits_l lp_l-1\big)\Big)\frac{\delta F}{\delta\mu}\big(\lambda\mu_{\bar{t}^n(r)}+(1-\lambda)\mu_{t^n(r)}\big)\big(r,\cdot,\mu_r\big),\,\mu_r\bigg\rangle\]
for each $n\geq1,\,(r,\lambda)\in[0,t)\times[0,1]$.
By Definition~\ref{C^2} and Assumption~\ref{assum:coefficients}, $\{g_n\}_{n\geq 1}$ is uniformly bounded.
Note that 
\[\lim_{n\to\infty}t^n(r)=\lim_{n\to\infty}\bar{t}^n(r)=r.\]
Hence, applying Definition~\ref{C^2} and Lemma~\ref{time_continuity} to $\{g_n\}_{n\geq 1}$, we have 
\[\lim_{n\to\infty}g_n(r,\lambda)=\Big\langle \cL \frac{\delta F}{\delta \mu}+\gamma\Big(\sum\limits_l lp_l-1\Big)\frac{\delta F}{\delta \mu} ,\mu_r\Big\rangle\]
for all $(r,\lambda)\in[0,t)\times[0,1]$.
By Lebesgue's dominated convergence theorem, we therefore obtain
\begin{equation}\label{ito_time_homogeneous}
    F(\mu_t)-F(\mu_0)=\int_0^t\Big\langle \cL \frac{\delta F}{\delta \mu}+\gamma\Big(\sum\limits_l lp_l-1\Big)\frac{\delta F}{\delta \mu} ,\mu_r\Big\rangle\,\d r
\end{equation}
for each $t\in[0,T]$.

\emph{Step 2}: For $F\in\cC_b^{1,2}\big([0,T]\times\cM_2(\R^d)\big)$, we denote by $F_t(\cdot):=F(t,\cdot)$ for all $t\in[0,T]$.
For $t\in[0,T]$, let us consider the same partition of $[0,t]$ as in \emph{Step 1},
and we have, for each $1\leq i\leq n$,
\[
\begin{aligned}
    &\quad F(t_i,\mu_{t_i})-F(t_{i-1},\mu_{t_{i-1}})\\
    =&\,F(t_i,\mu_{t_i})-F(t_{i-1},\mu_{t_i})+F_{t_{i-1}}(\mu_{t_i})-F_{t_{i-1}}(\mu_{t_{i-1}})\\
    =&\int_{t_{i-1}}^{t_i}\partial_r F(r,\mu_{t_i})\,\d r\\
    &\,+\int_{t_{i-1}}^{t_i}\Big\langle \cL \frac{\delta F_{t_{i-1}}}{\delta \mu}+\gamma\Big(\sum\limits_l lp_l-1\Big)\frac{\delta F_{t_{i-1}}}{\delta \mu} ,\mu_r\Big\rangle\,\d r,
\end{aligned}\]
where we apply \eqref{ito_time_homogeneous} to $F_{t_{i-1}}$ in the last line.
Hence, it holds
\[
\begin{aligned}
    &\quad F(t,\mu_t)-F(0,\mu_0)\\
    =&\int_0^t\partial_r F(r,\mu_{\bar{t}^n(r)})\,\d r\\
    &\,+\int_0^t\Big\langle \cL \frac{\delta F_{t^n(r)}}{\delta \mu}+\gamma\Big(\sum\limits_l lp_l-1\Big)\frac{\delta F_{t^n(r)}}{\delta \mu} ,\mu_r\Big\rangle\,\d r.
\end{aligned}\]
Analogously to \emph{Step 1}, we let $n\to\infty$, applying Definition~\ref{C^1,2} and Assumption~\ref{assum:coefficients}, and the proof of \eqref{ito_environment} is complete.

\subsection{It\^o's formula for $\mu_t^N$}
In this subsection, we prove \eqref{ito_empirical}.
We write
\[m(t):=\#\Big\{s\in(0,t]:\,\sup_{1\leq i\leq N}\sup_{k\in K_s^i}Q^{i,k}\big(\{s\}\times[0,\gamma(s,X_s^{i,k},\mu_{s-}^N)]\times[0,1]\big)=1\Big\},\]
and consider the stopping times
\[s_m:=\inf\{t\in[0,T]:\, m(t)\geq m\},\quad\forall\,m\geq0,\]
which characterize the jump moments.
Note that it holds almost surely
\[m(T)<\infty,\quad s_{m(T)}\leq T<\infty=s_{m(T)+1},\quad s_{m}<s_{m+1}\ \text{for }m\leq m(T).\]
Also note that $K_t^i=K_{s_m}^i$ for all $t\in[s_m,s_{m+1}),\,1\leq i\leq N$.

\emph{Step 1}: For $t\in[0,T]$, there exists $0\leq m\leq m(T)$ such that $t\in[s_m,s_{m+1})$.
We aim to prove
\begin{equation}\label{ito_nojump}
    \begin{aligned}
        F(t,\mu_t^N)=&F(s_m,\mu_{s_m}^N)+\int_{s_m}^t\bigg(\partial_s F+\Big\langle\cL\frac{\delta F}{\delta\mu}(s,\cdot,\mu_s^N),\mu_s^N\Big\rangle\bigg)\,\d s\\
        &\,+\frac{1}{2N}\int_{s_m}^t\d s\int_{\R^d}\trace\big(D_\mu^2 F(s,\mu_s^N,x,x)\sigma\sigma^\transpose(s,x,\mu_s^N)\big)\,\mu_s^N(\d x)\\
        &\,+\frac{1}{N}\sum\limits_{i=1}^N\int_{s_m}^t\sum\limits_{k\in K_{s_m}^i} D_\mu F(s,\mu_s^N,X_s^{i,k})\cdot\sigma(s,X_s^{i,k},\mu_s^N)\,\d W_s^{i,k}.
    \end{aligned}
\end{equation}
Note that $F$ and $(K_{s_m}^i)_{i=1}^N$ (we abbreviate it to $K_{s_m}$ in this subsection) naturally induces the function 
\[F^{K_{s_m}}\in\cC_b^{1,2}\Big([s_m,t]\times\prod\limits_{i=1}^N\prod\limits_{k\in K_{s_m}^i}\R^d\Big)\]
such that
\[F^{K_{s_m}}\big(s,(x^{i,k})\big)=F\Big(s,\frac{1}{N}\sum_{i=1}^N\sum_{k\in K_{s_m}^i}\delta_{x^{i,k}}\Big).\]
Hence, $F(s,\mu_s^N)$ can be rewritten as
\begin{equation}\label{classical_transformation}
    F(s,\mu_s^N)=F^{K_{s_m}}\big(s,(X_s^{i,k})\big)
\end{equation}
for all $s\in[s_m,t]$,
and we have
\begin{align*}
    \partial_s F^{K_{s_m}}\big(s,(X_s^{i,k})\big)=&\,\partial_s F(s,\mu_s^N),\\
    D_{x^{i,k}}F^{K_{s_m}}\big(s,(X_s^{i,k})\big)=&\,\frac{1}{N}D_\mu F(s,\mu_s^N,X_s^{i,k}),\\
    D_{x^{i,k}}^2 F^{K_{s_m}}\big(s,(X_s^{i,k})\big)=&\,\frac{1}{N}D_x ^2\frac{\delta F}{\delta\mu}(s,\mu_s^N,X_s^{i,k})\\
    &\,+\frac{1}{N^2}D_\mu^2 F(s,\mu_s^N,X_s^{i,k},X_s^{i,k}),\\
    D_{x^{i^\prime,k^\prime}}D_{x^{i,k}}F^{K_{s_m}}\big(s,(X_s^{i,k})\big)=&\,\frac{1}{N^2}D_\mu^2 F(s,\mu_s^N,X_s^{i,k},X_s^{i^\prime,k^\prime}),\quad(i,k)\neq(i^\prime,k^\prime).
\end{align*}
Applying the classical It\^o's formula to \eqref{classical_transformation}, we obtain \eqref{ito_nojump}.

\emph{Step 2}: For $t\geq0$, we aim to prove
\begin{equation}\label{ito_jump}
    \begin{aligned}
        &\quad \sum_{m=1}^{m(t)}\big(F(s_m,\mu_{s_m}^N)-F(s_m,\mu_{s_m-}^N)\big)\\
        =&\int_0^t\sum_{i=1}^N\sum_{k\in K_{s-}^i}\sum_{l\geq0}\Big(F\big(s,\mu_{s-}^N+\frac{l-1}{N}\delta_{X_s^{i,k}}\big)-F\big(s,\mu_{s-}^N\big)\Big)\,Q^{i,k}(\d s\times[0,\gamma]\times I_l),
    \end{aligned}
\end{equation}
where $[0,\gamma]\times I_l=([0,\gamma]\times I_l)(s,X_s^{i,k},\mu_{s-}^N)$.
Note that
\[
\begin{aligned}
    &\mu_t^N-\mu_{t-}^N=\sum_{i=1}^N\sum_{k\in K_{t-}^i}\sum_{l\geq0}\frac{l-1}{N}\delta_{X_t^{i,k}}\,Q^{i,k}(\{t\}\times[0,\gamma]\times I_l).
\end{aligned}\]
Also note that for $i\neq i^\prime$ or $k\neq k^\prime$, Poisson random measures $Q^{i,k}$ and $Q^{i^\prime,k^\prime}$ are independent; for $l\neq l^\prime$, $I_l$ and $I_{l^\prime}$ does not intersect. 
Hence, it holds almost surely that only one particle dies and gives birth to $l$ new particles at some $t=s_m$.
Therefore, we have, almost surely,
\[F(t,\mu_t^N)=\sum_{i=1}^N\sum_{k\in K_{t-}^i}\sum_{l\geq0}F\big(t,\mu_{t-}^N+\frac{l-1}{N}\delta_{X_t^{i,k}}\big)\,Q^{i,k}(\{t\}\times[0,\gamma]\times I_l),\]
which leads to \eqref{ito_jump}.

\emph{Step 3}: For $t\in[0,T]$, it holds
\[
\begin{aligned}
    F(t,\mu_t^N)-F(0,\mu_0^N)=&F(t,\mu_t^N)-F(s_m(t),\mu_{s_{m(t)}}^N)+\sum_{m=1}^{m(t)}\big(F(s_m,\mu_{s_m-}^N)-F(s_{m-1},\mu_{s_{m-1}}^N)\big)\\
    &\quad+\sum_{m=1}^{m(t)}\big(F(s_m,\mu_{s_m}^N)-F(s_m,\mu_{s_m-}^N)\big).
\end{aligned}\]
Applying \eqref{ito_nojump} to
\[F(t,\mu_t^N)-F(s_{m(t)},\mu_{s_{m(t)}}^N)+\sum_{m=1}^{m(t)}\big(F(s_m,\mu_{s_m-}^N)-F(s_{m-1},\mu_{s_{m-1}}^N)\big),\]
and \eqref{ito_jump} to
\[\sum_{m=1}^{m(t)}\big(F(s_m,\mu_{s_m}^N)-F(s_m,\mu_{s_m-}^N)\big),\]
we obtain \eqref{ito_empirical}.
Note that we replace $Q^{i,k}(\d t\d z)$ with $\bar{Q}^{i,k}(\d t,\d z)+\d t\d z$,
and use
\[\mu_t^N=\mu_{t-}^N\quad\text{a.e.\quad in\quad} \Omega\times[0,T]\]
in \eqref{ito_empirical}.

The proof of \eqref{ito_empirical} is complete.
\section{Proof of Lemma~\ref{lifted_wellposedness}}\label{lifted_wellposedness_proof}
The idea of proof is derived from \cite[Theorem 5.2.1]{oksendal2013stochastic}, where we define $\{(Y^{(k)},Z^{(k)})\}_{k=0}^\infty$ as follows.
Let $(Y_t^{(0)},Z_t^{(0)})= (Y_0,Z_0)$ for all $t\in[0,T]$, and inductively,
    \[
    \left\{
    \begin{aligned}
        Y_t^{(k+1)}=\,&Y_0+\int_0^t b\big(s,Y_s^{(k)},\cT^*\rho^{(k)}_s\big)\,\d s+\int_0^t \sigma\big(s,Y_s^{(k)},\cT^*\rho^{(k)}_s\big)\,\d W_s,\\
        Z_t^{(k+1)}=\,&Z_0+\int_0^t c\big(s,Y_s^{(k)},\cT^*\rho^{(k)}_s\big) Z_s^{(k)}\,\d s,
    \end{aligned}
    \right.
    \]
for all $t\in[0,T]$, where $\rho^{(k)}_t:=\text{Law}(Y_t^{(k)},Z_t^{(k)})$.
Note that $Z_0\in[0,C]$, and hence we can inductively prove, by Assumption~\ref{assum:coefficients}, that 
\[Z_t^{(k)}\in [0,C\exp(\bar{\gamma}Mt)],\]
and thus
\[\rho_t^{(k)}\in\cP_2\big(\R^d\times [0,C\exp(\bar{\gamma}MT)]\big),\]
for all $t\in[0,T],\,k\geq 0.$
Hence, applying Assumption~\ref{assum:coefficients} again, we have
\[
\begin{aligned}
    &\quad\Expect \Big[\sup_{t\in[0,T]}\abs{Y_t^{(k+1)}-Y_t^{(k)}}^2\Big]\\
    =\,&\Expect\Big[\sup_{t\in[0,T]}\Big\vert \int_0^t \big(b\big(s,Y_s^{(k)},\cT^*\rho^{(k)}_s)-b\big(s,Y_s^{(k-1)},\cT^*\rho^{(k-1)}_s)\big)\,\d s\\
    &+\int_0^t\big(\sigma\big(s,Y_s^{(k)},\cT^*\rho^{(k)}_s)-\sigma\big(s,Y_s^{(k-1)},\cT^*\rho^{(k-1)}_s)\big)\,\d W_s\Big\vert^2\Big]\\
    \leq\,& 2\Expect\Big[\sup_{t\in[0,T]}\Big\vert \int_0^t \big(b\big(s,Y_s^{(k)},\cT^*\rho^{(k)}_s)-b\big(s,Y_s^{(k-1)},\cT^*\rho^{(k-1)}_s)\big)\,\d s\Big\vert^2\\
    &+\sup_{t\in[0,T]}\Big\vert\int_0^t\big(\sigma\big(s,Y_s^{(k)},\cT^*\rho^{(k)}_s)-\sigma\big(s,Y_s^{(k-1)},\cT^*\rho^{(k-1)}_s)\big)\,\d W_s\Big\vert^2\Big]\\
    \leq\,& 2\Expect\Big[T \int_0^T L^2\Big(\big\vert Y_t^{(k)}-Y_t^{(k-1)}\big\vert+\mathbf{d}(\cT^*\rho^{(k)}_t,\cT^*\rho^{(k-1)}_t) \Big)^2\,\d t\Big]\\
    &+2C_2\Expect\Big[\int_0^T\norm{\sigma\big(s,Y_s^{(k)},\cT^*\rho^{(k)}_s)-\sigma\big(s,Y_s^{(k-1)},\cT^*\rho^{(k-1)}_s)}^2\,\d s\Big],
\end{aligned}\]
where we use the Cauchy--Schwarz inequality and Burkholder--Davis--Gundy inequality in the last step, and $C_2$ is a positive constant.
By Lemma~\ref{Lipschitz_lifting} and definition of Wasserstein distance, it holds that
\[
\begin{aligned}
    &\quad\mathbf{d}(\cT^*\rho^{(k+1)}_t,\cT^*\rho^{(k)}_t)\leq (C\exp(\bar{\gamma}MT)+1)\cW_2(\rho^{(k)}_t,\rho^{(k-1)}_t)\\
    \leq\,& (C\exp(\bar{\gamma}MT)+1)\Expect\Big[\abs{Y_t^{(k)}-Y_t^{(k-1)}}^2+\abs{Z_t^{(k)}-Z_t^{(k-1)}}^2\Big]^{1/2},
\end{aligned}\]
and thus
\[
\begin{aligned}
    &\quad\Expect \Big[\sup_{t\in[0,T]}\abs{Y_t^{(k+1)}-Y_t^{(k)}}^2\Big]\\
    \leq\,& 4L^2 T\Expect\Big[\int_0^T\Big(\big\vert Y_t^{(k)}-Y_t^{(k-1)}\big\vert^2\\
    &\,+(C\exp(\bar{\gamma}MT)+1)^2\Expect\big[\big\vert Y_t^{(k)}-Y_t^{(k-1)}\big\vert^2+\big\vert Z_t^{(k)}-Z_t^{(k-1)}\big\vert^2\big]\Big)\,\d t\Big]\\
    &\,+4C_2 L^2 \Expect\Big[\int_0^T\Big(\big\vert Y_t^{(k)}-Y_t^{(k-1)}\big\vert^2\\
    &\,+(C\exp(\bar{\gamma}MT)+1)^2\Expect\big[\big\vert Y_t^{(k)}-Y_t^{(k-1)}\big\vert^2+\big\vert Z_t^{(k)}-Z_t^{(k-1)}\big\vert^2\big]\Big)\,\d t\Big]\\
    \leq\,& D \Expect\Big[\int_0^T \Big(\abs{Y_t^{(k)}-Y_t^{(k-1)}}^2+\abs{Z_t^{(k)}-Z_t^{(k-1)}}^2\Big)\,\d t\Big],
\end{aligned}\]
where $D\in(0,\infty)$ denotes a constant dependent only on $M,T,\bar{\gamma},C,C_2,L$,
whose value can vary through lines.
Analogously, we have
\[
\begin{aligned}
    &\quad\Expect \Big[\sup_{t\in[0,T]}\abs{Z_t^{(k+1)}-Z_t^{(k)}}^2\Big]\\
    \leq\,& T\Expect\Big[\int_0^T\Big\vert\big(c(t,Y_t^{(k)},\cT^*\rho_t^{(k)})-c(t,Y_t^{(k-1)},\cT^*\rho_t^{(k-1)})\big)Z_t^{(k)}\\
    &\,+c(t,Y_t^{(k-1)},\cT^*\rho_t^{(k-1)})(Z_t^{(k)}-Z_t^{(k)})\Big\vert^2\,\d t\Big] \\
    \leq\,& D \Expect\Big[\int_0^T \Big(\abs{Y_t^{(k)}-Y_t^{(k-1)}}^2+\abs{Z_t^{(k)}-Z_t^{(k-1)}}^2\Big)\,\d t\Big].
\end{aligned}\]
Hence, let us consider the Banach space $\mathbb{H}^2_{d+1}:=\mathbb{L}^2(\Omega;\cC([0,T];\R^{d+1}))$,
where 
\[\norm{U}_{\mathbb{H}^2_{d+1}}:=\Expect\Big[\sup_{t\in[0,T]}\abs{U_t}^2\Big]^{1/2},\quad\forall U\in \mathbb{H}^2_{d+1},\]
and we have
\[\norm{(Y,Z)^{(k+1)}-(Y,Z)^{(k)}}_{\mathbb{H}^2_{d+1}}^2\leq D \int_0^T \Expect\Big[\abs{Y_t^{(k)}-Y_t^{(k-1)}}^2+\abs{Z_t^{(k)}-Z_t^{(k-1)}}^2\Big]\,\d t\]
for all $k\geq 0$.
Note that the argument can directly yield
\[
\begin{aligned}
    &\quad\Expect\Big[\abs{Y_t^{(k+1)}-Y_t^{(k)}}^2+\abs{Z_t^{(k+1)}-Z_t^{(k)}}^2\Big]\\
    \leq\,&D\int_0^t \Expect\Big[\abs{Y_s^{(k+1)}-Y_s^{(k)}}^2+\abs{Z_s^{(k+1)}-Z_s^{(k)}}^2\Big]\,\d s
\end{aligned}\]
for all $k\geq0,t\in[0,T]$,
and it is clear that
\[\Expect\Big[\abs{Y_t^{(1)}-Y_t^{(0)}}^2+\abs{Z_t^{(1)}-Z_t^{(0)}}^2\Big]\leq\norm{(Y,Z)^{(1)}-(Y,Z)^{(0)}}_{\mathbb{H}^2_{d+1}}^2\leq D.\]
Now we fix $D$, and can inductively prove that
\[\Expect\Big[\abs{Y_t^{(k)}-Y_t^{(k-1)}}^2+\abs{Z_t^{(k)}-Z_t^{(k-1)}}^2\Big]\leq D\frac{(Dt)^{k-1}}{(k-1)!},\]
and consequently,
\[\norm{(Y,Z)^{(k)}-(Y,Z)^{(k-1)}}_{\mathbb{H}^2_{d+1}}^2\leq D\frac{(DT)^{k-1}}{(k-1)!}.\]
Hence, it holds that
\[
\begin{aligned}
    &\quad\norm{(Y,Z)^{(0)}}_{\mathbb{H}^2_{d+1}}+\sum_{k=1}^\infty\norm{(Y,Z)^{(k)}-(Y,Z)^{(k-1)}}_{\mathbb{H}^2_{d+1}}\\
    \leq\,& \norm{(Y,Z)^{(0)}}_{\mathbb{H}^2_{d+1}}+\sum_{k=1}^\infty\sqrt{D\frac{(DT)^{k-1}}{(k-1)!}}<\infty,
\end{aligned}
\]
and thus, by completeness,
\[(Y,Z):=(Y,Z)^{(0)}+\sum_{k=1}^\infty\big((Y,Z)^{(k)}-(Y,Z)^{(k-1)}\big)\in\mathbb{H}^2_{d+1},\]
with
\[
\begin{aligned}
    \norm{(Y,Z)^{(k)}-(Y,Z)}_{\mathbb{H}^2_{d+1}}\leq\,&\sum_{n=k+1}^\infty\norm{(Y,Z)^{(n)}-(Y,Z)^{(n-1)}}_{\mathbb{H}^2_{d+1}}\to0\\
    \sup_{t\in[0,T]}\cW_2(\rho_t^{(k)},\rho_t)\leq\,&\norm{(Y,Z)^{(k)}-(Y,Z)}_{\mathbb{H}^2_{d+1}}^2\to 0,
\end{aligned}\]
as $k\to\infty$, where $\rho_t=\text{Law}(Y_t,Z_t)$.
On the other hand, applying Chebyshev's inequality we have
\[
\begin{aligned}
    &\quad\sum_{k=1}^\infty\dbP\Big(\sup_{t\in[0,T]}\sqrt{\big\vert Y_t^{(k)}-Y_t^{(k-1)}\big\vert^2+\big\vert Z_t^{(k)}-Z_t^{(k-1)}\big\vert^2}\geq 2^{-k}\Big)\\
    \leq\,&\sum_{k=1}^\infty 2^{k}\norm{(Y,Z)^{(k)}-(Y,Z)^{(k-1)}}_{\mathbb{H}^2_{d+1}}<\infty.
\end{aligned}
\]
Consequently, there exists $N=N(\omega)$ for almost all $\omega\in\Omega$, by the Borel--Cantelli lemma, such that
\[\sup_{t\in[0,T]}\sqrt{\big\vert Y_t^{(k)}-Y_t^{(k-1)}\big\vert^2+\big\vert Z_t^{(k)}-Z_t^{(k-1)}\big\vert^2}< 2^{-k}\]
for all $k\geq N$,
and thus $(Y,Z)^{(k)}$ converges to $(Y,Z)$ almost surely.
Hence, we can verify that $(Y,Z)$ is a strong solution to the initial value problem for \eqref{lifted_MeanField}, and thus prove existence.

For strong solutions $(Y,Z)$ and $(\hat{Y},\hat{Z})$, it holds analogously that
\[\Expect\Big[\abs{Y_t-\hat{Y}_t}^2+\abs{Z_t-\hat{Z}_t}^2\Big]\leq \tilde{D}\int_0^t \Expect\Big[\abs{Y_s-\hat{Y}_s}^2+\abs{Z_s-\hat{Z}_s}^2\Big]\,\d s\]
for all $t\in[0,T]$, where $\tilde{D}\in (0,\infty)$ is a constant dependent only on $M,T,\bar{\gamma},C,C_2,L$.
Hence, we have, by Gronwall's inequality,
\[\Expect\Big[\abs{Y_t-\hat{Y}_t}^2+\abs{Z_t-\hat{Z}_t}^2\Big]\leq 0\cdot\exp(\tilde{D}t)=0\]
for all $t\in[0,T]$,
and thus
\[\dbP\big(Y_t=\hat{Y}_t,\,Z_t=\hat{Z}_t,\quad\forall t\in \mathbb{Q}\cap [0,T]\big)=1.\]
By continuity of paths, it holds that
\[(Y,Z)=(\hat{Y},\hat{Z})\quad \text{a.s.},\]
which yields uniqueness.

Combining the above, the proof of Lemma~\ref{lifted_wellposedness} is complete.

\section{Nonlinear Fokker--Planck equation and Flow property}\label{append_flow}
In this section, we study the Fokker--Planck equation associated with the branching diffusion \eqref{solution},
in order to establish the flow property for environment measures $(\mu_t)_{t\in[0,T]}$ of \eqref{solution} and prove Lemma~\ref{lem:mollify}.

\subsection{The nonlinear Fokker--Planck equation}

Let Assumptions~\ref{assum:coefficients} and \ref{assum:nondegen} hold, and we consider $t\in [0,T)$ and $\mu\in\cM_2(\R^d)$.
For each $m\in\cP_1(E)$ such that $\Psi m=\mu$, we write
\[\mu_s:=\Psi m_s^{t,m}\in\cM_2(\R^d),\quad s\in [t,T]\]
Recall from Lemma \ref{lem:FKP} that the flow of measure $(\mu_t)_{t \in [0,T]}$ satisfies:
for each $s\in[t,T]$ and $f\in\cC^{1,2}_b([t,T]\times\R^d)$,
\begin{equation} \label{F-P_duality_append}
    \langle f_s,\mu_s\rangle
        =\langle f_t,\mu\rangle+\int_t^s\Big\langle \partial_r f_r(\cdot)+\Big(\cL f_r+\gamma\Big(\sum\limits_l lp_l-1\Big)f_r\Big)(r,\cdot,\mu_r),\mu_r\Big\rangle\,\d r.
\end{equation}
Let $u=(\mu_s)_{s\in[t,T]}\in\cC([t,T];\cM_2(\R^d))$ (recall Lemma~\ref{time_continuity} for the time-continuity of environment measures),
we say that $u$ is a weak solution to the nonlinear parabolic Fokker--Planck equation
\begin{equation}\label{F-P}
    \left\{
    \begin{aligned}
        \partial_s u=\,&\Big(\cL^*+\gamma \big(\sum_l lp_l-1\big)\Big)u,\quad t\in(t,T],\\
        u(t)=\,&\mu
    \end{aligned}
    \right.
\end{equation}
in the sense of \eqref{F-P_duality_append},
where $\cL^*$ denotes the dual of the infinitesimal generator given by \eqref{infititesimal}.
If we impose some density conditions on the initial value, well-posedness of \eqref{F-P} can be obtained as follows (cf.~\cite[Proposition 2.4]{hambly2025optimal}; see also~\cite[Theorem A.1]{claisse2025optimal}).
\begin{prop}\label{prop:uniqueness}
   Let Assumptions~\ref{assum:coefficients} and \ref{assum:nondegen} hold.
   Suppose a weak derivative $(\partial_x\sigma\sigma^\transpose)(t,\cdot,\nu)\in\mathbb{L}_{\text{loc}}^1(\R^d;\R^{d\times d})$ exists for each $(t,\nu)\in [0,T]\times\cM_2(\R^d)$ and that $\mu\in\cM_2^\theta(\R^d)$ for some $\theta>0$ (recall Definition~\ref{density_regularity}).
   Then the Fokker--Planck equation \eqref{F-P} admits a weak unique solution 
   \[u=(\mu_s)_{s\in[t,T]}\in\cC([t,T];\cM_2(\R^d)),\]
   and we have $\mu_s\in\cM_2^\theta(\R^d)$ for each $s\in [t,T]$.
\end{prop}
\begin{rem}
    It is clear that existence of the weak derivative of $\sigma\sigma^\transpose(t,\cdot,\nu)$ is ensured by Assumption~\ref{assum:differentiability}.
    For a more general $\sigma$, this can be fulfilled through mollification (see Subsection~\ref{proof:mollify}).
\end{rem}
Combined with flow property of $m_t$, uniqueness of the weak solution directly leads to the following result.
\begin{cor}\label{cor:flow}
    Suppose Assumptions~\ref{assum:coefficients} and \ref{assum:nondegen} hold,
    and that a weak derivative $(\partial_x\sigma\sigma^\transpose)(t,\cdot,\nu)\in\mathbb{L}_{\text{loc}}^1(\R^d;\R^{d\times d})$ exists for each $(t,\nu)\in [0,T]\times\cM_2(\R^d)$.
    Then for each $t\in [0,T],\theta>0$, the mapping
    \begin{equation}\label{flow_branching_environment}
        \begin{aligned}
        \mu_\cdot^{t,\cdot}:[t,T]\times\cM_2^\theta(\R^d)\,&\to\cM_2^\theta(\R^d)\\
        (s,\mu)\,&\mapsto\mu_s^{t,\mu}:=\Psi m_s^{t,m},\,m\in\Psi^{-1}(\{\mu\})
    \end{aligned}
    \end{equation}
     is well-defined.
     Furthermore, it holds that
     \[\mu_s^{t,\mu_t}=\mu_s,\,\mu_s^{r,\mu_r^{t,\mu}}=\mu_s^{t,\mu}\]
     for each $0\leq t\leq r\leq s\leq T,\,\mu\in\cM_2^\theta(\R^d)$.
\end{cor}

\subsection{Proof of Lemma~\ref{lem:mollify}}\label{proof:mollify}
    Recall $\tilde{\mu_t}$ in \eqref{lifted_MeanField}, and we observe that
    $(\tilde{\mu_t})_{t\in[0,T]}$ also induces a weak solution to \eqref{F-P}.
    Lemma~\ref{lem:mollify}, however, does not assume existence of either density of initial distributions or a weak derivative of $\sigma\sigma^\transpose(t,\cdot,\mu)$.
    Hence, the key idea is to truncate and mollify \eqref{lifted_MeanField} to ensure uniqueness of the solution to its corresponding Fokker--Planck equation.
    Note that similar methods were used in \cite[Proposition A.2]{claisse2025optimal}.
    
    Without loss of generality, we assume $t=0$ and consider strong solutions $(Y,Z),\,(\hat{Y},\hat{Z})$ of \eqref{lifted_MeanField} with marginal distributions
    \[\text{Law}(Y_t,Z_t)=:\rho_t,\quad\text{Law}(\hat{Y}_t,\hat{Z}_t)=:\hat{\rho}_t,\quad t\in[0,T]\]
    such that $\rho_0=\rho,\,\hat{\rho}_0=\hat{\rho}$.
    It suffices to show $\cT^*\rho_t=\cT^*\hat{\rho}_t$ for each $t\in[0,T]$.
    For $\varepsilon>0$, let $\{\alpha^\varepsilon(\cdot)\}_{\varepsilon>0}\subset\cC_c^\infty(\R^d)$ a classical family of mollifiers such that 
    \[\alpha^\varepsilon\geq 0,\quad\int_{\R^d}\alpha^\varepsilon(x)\,\d x=1,\quad\text{supp}(\alpha^\varepsilon)=\{x\in\R^d:\abs{x}\leq \varepsilon\}.\]
    We mollify $\sigma$ through the component-wise convolution
    \[\sigma^\varepsilon(t,\cdot,\nu):=\sigma(t,\cdot,\nu)*\alpha^\varepsilon:=\int_{\R^d}\alpha^\varepsilon(y)\sigma(t,\cdot-y,\nu)\,\d y\]
    for each $(t,\nu)\in[0,T]\times\cM_2(\R^d)$.
    It is clear that $\{\sigma^\varepsilon(\cdot)\}_{\varepsilon>0}$ are bounded Lipschitz as required by Assumption~\ref{assum:coefficients}, uniformly for $\varepsilon>0$, and $\sigma^\varepsilon(t,\cdot,\nu)$ is smooth for each $\varepsilon>0$, which ensures existence of a partial derivative.
    We also observe, for each $(t,\nu)\in[0,T]\times\cM_2(\R^d)$, (recall the Lipschitz constant $L$ in Assumption~\ref{assum:coefficients})
    \[\norm{\sigma^\varepsilon(t,\cdot,\nu)-\sigma(t,\cdot,\nu)}_\infty\leq L\varepsilon\to0,\quad\varepsilon\to0,\]
    and that $\sigma^\varepsilon(\sigma^\varepsilon)^\transpose(\cdot)$ is still uniformly non-degenerate as Assumption~\ref{assum:nondegen}, for each fixed small enough $\varepsilon>0$.
    On the other hand, we consider a standard Gaussian variable $\xi$ with density $f_\xi$, independent of initial values $(Y_0,Z_0)$ and $(\hat{Y}_0,\hat{Z}_0)$.
    For $x=(x_1,...,x_d)\in\R^d,\,a\in\R$, we write
    \[x\wedge a:=(x_1\wedge a,...,x_d\wedge a),\quad x\vee a:=(x_1\vee a,...,x_d\vee a).\]
    Now we can construct approximating diffusions $(Y^\varepsilon,Z^\varepsilon)$ and $(\hat{Y}^\varepsilon,\hat{Z}^\varepsilon)$ with marginal distributions $(\rho_t^\varepsilon)$ and $(\hat{\rho}_t^\varepsilon)$, respectively.
    We set the initial values
    \[
    \begin{aligned}
        \left\{
    \begin{aligned}
        Y_0^\varepsilon&=\big(-\frac{1}{\varepsilon}\big)\vee \big(Y_0\wedge\frac{1}{\varepsilon}\big)+\varepsilon\xi,\\
        Z_0^\varepsilon&=Z_0,
    \end{aligned}
    \right.
    \quad
    \left\{
    \begin{aligned}
        \hat{Y}_0^\varepsilon&=\big(-\frac{1}{\varepsilon}\big)\vee \big( \hat{Y}_0\wedge\frac{1}{\varepsilon}\big)+\varepsilon\xi,\\
        \hat{Z}_0^\varepsilon&=\hat{Z}_0,
    \end{aligned}
    \right.
    \end{aligned}
    \]
    and establish SDEs
    \[
    \begin{aligned}
        \left\{
    \begin{aligned}
        \d Y_t^\varepsilon&=b(t,Y_t^\varepsilon,\cT^*\rho_t^\varepsilon)\,\d t+\sigma^\varepsilon(t,Y_t^\varepsilon,\cT^*\rho_t^\varepsilon)\,\d W_t,\\
        \d Z_t^\varepsilon&=c(t,Y_t^\varepsilon,\cT^*\rho_t^\varepsilon)Z_t^\varepsilon\,\d t,
    \end{aligned}
    \right.
    \\
    \left\{
    \begin{aligned}
        \d \hat{Y}_t^\varepsilon&=b(t,\hat{Y}_t^\varepsilon,\cT^*\hat{\rho}_t^\varepsilon)\,\d t+\sigma^\varepsilon(t,\hat{Y}_t^\varepsilon,\cT^*\hat{\rho}_t^\varepsilon)\,\d W_t,\\
        \d \hat{Z}_t^\varepsilon&=c(t,\hat{Y}_t^\varepsilon,\cT^*\hat{\rho}_t^\varepsilon)\hat{Z}_t^\varepsilon\,\d t,
    \end{aligned}
    \right.
    \end{aligned}
    \]
    whose well-posedness holds by Lemma~\ref{lifted_wellposedness}.
    For each $\varphi\in\cC_b(\R^d)$, we observe, by independence, that
    \[
    \begin{aligned}
        &\quad\Expect[\varphi(Y_0^\varepsilon)Z_0^\varepsilon]=\int_{\R^d}\d x\, f_\xi(x)\Expect\Big[\varphi\Big(\big(-\frac{1}{\varepsilon}\big)\vee \big(Y_0\wedge\frac{1}{\varepsilon}\big)+\varepsilon x\Big) Z_0\Big]\\
        =\,&\int_{\R^d}\d x\, f_\xi(x)\Expect\Big[\varphi\Big(\big(-\frac{1}{\varepsilon}\big)\vee \big(\hat{Y}_0\wedge\frac{1}{\varepsilon}\big)+\varepsilon x\Big) \hat{Z}_0\Big]
        =\Expect[\varphi(\hat{Y}_0^\varepsilon)\hat{Z}_0^\varepsilon],
    \end{aligned}\]
    which implies $\cT^* \rho_0^\varepsilon=\cT^*\hat{\rho}_0^\varepsilon$.
    Furthermore, let $g:\R^d\to\R$ be some Borel measurable function,
    which is almost everywhere bounded nonnegative,
    such that
    \[\Expect[Z_0^\varepsilon | Y_0^\varepsilon]=g(Y_0^\varepsilon),\]
    and we have
    \[
    \begin{aligned}
        \Expect[\varphi(Y_0^\varepsilon)Z_0^\varepsilon]=\,&\Expect\Big[\Expect[\varphi(Y_0^\varepsilon)Z_0^\varepsilon |Y_0^\varepsilon]\Big]=\Expect[\varphi(Y_0^\varepsilon)g(Y_0^\varepsilon)],
    \end{aligned}
    \]
    \ie~$\cT^* \rho_0^\varepsilon$ has a density
    \[\varrho^\varepsilon(y):=g(y)\varrho_Y(y),\quad y\in\R^d.\]
    Here we denote by
    \[\varrho_Y(y):=\int_{\R^d}\frac{1}{\varepsilon^d}f_\xi(\frac{y-y^\prime}{\varepsilon})\rho_{(Y_0,\varepsilon)}(\d y^\prime),\quad y\in\R^d \]
    the density of $Y_0^\varepsilon$,
    for the marginal distribution 
    \[\rho_{(Y_0,\varepsilon)}:=\text{Law}\Big(\big(-\frac{1}{\varepsilon}\big)\vee \big( Y_0\wedge\frac{1}{\varepsilon}\big)\Big).\]
    Therefore, we observe, for each fixed $\varepsilon>0$, that
    \[\varrho^\varepsilon(y)\leq C_\varepsilon\exp\big(-\frac{\abs{y}^2}{2C_\varepsilon}\big),\quad\text{a.e. }y\in\R^d,\]
    for some $C_\varepsilon\in(0,\infty)$,
    which implies $\cT^*\rho_0^\varepsilon\in\cM_2^\theta(\R^d)$ for some $\theta>0$.
    Note that $(\cT^*\rho_t^\varepsilon)_{t\in[0,T]}$ and $(\cT^*\hat{\rho}_t^\varepsilon)_{t\in[0,T]}$ are weak solutions to the same nonlinear Fokker--Planck equation of the form \eqref{F-P}, with $\sigma$ replaced by $\sigma^\varepsilon$.
    Hence, applying Proposition~\ref{prop:uniqueness} we have
    \[\cT^*\rho_t^\varepsilon=\cT^*\hat{\rho}_t^\varepsilon,\quad t\in[0,T],\]
    for all small enough $\varepsilon>0$.

    On the other hand, through similar techniques to those used in Section~\ref{lifted_wellposedness_proof} for uniqueness,
    we can obtain
    \[\Expect\big[\abs{Y_t^\varepsilon-Y_t}^2+\abs{Z_t^\varepsilon-Z_t}^2\big]\leq C\Expect\big[\abs{Y_0^\varepsilon-Y_0}^2\big],\quad t\in[0,T],\]
    for some constant $C\in(0,\infty)$ independent of $\varepsilon$ and $t$,
    which implies
    \[\cW_2(\rho_t^\varepsilon,\rho_t)\to 0,\quad\varepsilon\to0^+,\]
    for each $t\in[0,T]$.
    Analogously, this holds for $(\hat{\rho}_t^\varepsilon)_{t\in[0,T]}$,
    and thus, by Lemma~\ref{Lipschitz_lifting},
    \[\cT^*\rho_t=\lim_{\varepsilon\to0^+}\cT^*\rho_t^\varepsilon=\lim_{\varepsilon\to0^+}\cT^*\hat{\rho}_t^\varepsilon=\cT^*\hat{\rho}_t.\]

    The proof is complete.

\section{Proof of Lemma~\ref{intrinsic_is_linear}}\label{proof_intrinsic_is_linear}
Let us firstly introduce the notation $(Y,Z)_\cdot^{t,(y,z),\rho}$ derived from \cite{10.1214/15-AOP1076}, namely,
\[
 \left\{
    \begin{aligned}
        Y_s^{t,(y,z),\rho}&=y+\int_t^s b(r,Y_r^{t,(y,z),\rho},\tilde{\mu}_r^{t,\mu})\,\d r+\int_t^s\sigma(r,Y_r^{t,(y,z),\rho},\tilde{\mu}_r^{t,\mu})\,\d W_r,\\
        Z_s^{t,(y,z),\rho}&=z+\int_t^s c(r,Y_r^{t,(y,z),\rho},\tilde{\mu}_r^{t,\mu})Z_r^{t,(y,z),\rho}\,\d r,
    \end{aligned}
    \right.\]
for $s\in[t,T]$, $(y,z)\in\R^d\times[0,C],\,\rho\in\cP_2(\R^d\times[0,C])$.
Here $\mu=\cT^*\rho$ and (recall \eqref{flow_lifting_defn}) $\tilde{\mu}_r^{t,\mu}=\cT^*\rho_r^{t,\rho},\,r\in[t,T]$.
By \cite[Lemma 6.1]{10.1214/15-AOP1076} and Lemma \ref{regularity_G_lifting}, it holds that
\begin{equation}\label{intrinsic_derivative_lifted}
    \begin{aligned}
        &\quad(D_\rho \tilde{U})_j(t,\rho,(y,z))=\,\sum_{i=1}^{d+1}\Expect\Big[(D_\rho\tilde{G})_i(\rho_T^{t,\rho},(Y,Z)_T^{t,(y,z),\rho})\cdot\partial_{y_j} (Y,Z)_{T,i}^{t,(y,z),\rho}\\
        &\,+(D_\rho\tilde{G})_i(\rho_T^{t,\rho},(Y,Z)_T^{t,\xi})\cdot(D_\rho (Y,Z)_{T,i}^{t,\xi,\rho})_{j}(y,z)\Big]\\
        =&\,\sum_{i=1}^d\Expect\Big[(D_\mu G)_i(\tilde{\mu}_T^{t,\mu},Y_T^{t,(y,z),\rho})Z_{T,i}^{t,(y,z),\rho}\cdot \partial_{y_j} Y_{T,i}^{t,(y,z),\rho}
        +(D_\mu G)_i(\tilde{\mu}_T^{t,\mu},Y_T^{t,\xi})Z_{T,i}^{t,\xi}\cdot(D_\rho Y_{T,i}^{t,\xi,\rho})_{j}(y,z)\Big]\\
        &\,+\Expect\Big[\frac{\delta G}{\delta\mu}(\tilde{\mu}_T^{t,\mu},Y_T^{t,(y,z),\rho})\cdot \partial_{y_j} Z_T^{t,(y,z),\rho}
        +\frac{\delta G}{\delta\mu}(\tilde{\mu}_T^{t,\mu},Y_T^{t,\xi})\cdot(D_\rho Z_T^{t,\xi,\rho})_{j}(y,z)\Big],\quad 1\leq j\leq d+1,
    \end{aligned}
\end{equation}
where $\text{Law}(\xi)=\rho,\,y_{d+1}=:z$.
Note that the intrinsic derivative of $\tilde{U}$ relies heavily on derivatives of $Y$ and $Z$.   
Hence, we turn back to \eqref{lifted_MeanField}.
Without loss of generality, we assume the coefficients $b,\sigma,c$ are time homogeneous.
Note that $b,\sigma$ and $c$ are independent of $z$, the $(d+1)$th variable.
Hence, by \cite[Lemma 4.1]{10.1214/15-AOP1076}, the first derivative $(\partial_z (Y,Z)_s^{t,(y,z),\rho})_{s\in[t,T]}$ is the unique solution of
the SDE
\begin{equation}
    \left\{
    \begin{aligned}
        \partial_z Y_{s,j}^{t,(y,z),\rho}=&\,\sum_{k=1}^d\int_t^s\partial_{y_k} b_j(Y_r^{t,(y,z),\rho},\tilde{\mu}_r^{t,\mu})\partial_z Y_{r,k}^{t,(y,z),\rho}\,\d r\\
            &\,+\sum_{k,\,l=1}^d\int_t^s\partial_{y_k} \sigma_{j,l}(Y_r^{t,(y,z),\rho},\tilde{\mu}_r^{t,\mu})\partial_z Y_{r,k}^{t,(y,z),\rho}\,\d W_r^{l},\\
            &\qquad\qquad\qquad\qquad\qquad\qquad\qquad\qquad\qquad \text{for}\quad j=1,...,d,\\
        \partial_z Z_s^{t,(y,z),\rho}=&\,1+\sum_{k=1}^d\int_t^s\partial_{y_k} c(Y_r^{t,(y,z),\rho},\tilde{\mu}_r^{t,\mu})\partial_z Y_{r,k}^{t,(y,z),\rho}\,\d r\\
        &\,+\int_t^s c(Y_r^{t,(y,z),\rho},\tilde{\mu}_r^{t,\mu})\partial_z Z_r^{t,(y,z),\rho}\,\d r
    \end{aligned}
    \right.
\end{equation}
and thus we have
\[\partial_z Y_{s,j}^{t,(y,z),\rho}=0,\quad \partial_z Z_s^{t,(y,z),\rho}=\exp\Big(\sum_{k=1}^d\int_t^s c(Y_r^{t,(y,z),\rho},\tilde{\mu}_r^{t,\mu})\,\d r\Big)\]
for each $s\in[t,T],\,j=1,...,d$.
Thus it holds
\begin{equation}\label{lifted_solution}
    Y_{s}^{t,(y,z),\rho}=Y_{s}^{t,y,\rho},\quad Z_s^{t,(y,z),\rho}=z\exp\Big(\sum_{k=1}^d\int_t^s c(Y_r^{t,y,\rho},\tilde{\mu}_r^{t,\mu})\,\d r\Big)=:zF_{s}^{t,y,\rho}.
\end{equation}
for all $(y,z)\in\R^d\times[0,C]$ and $\rho\in\cP_2(\R^d\times[0,C])$.

By \cite[Theorem 4.1]{10.1214/15-AOP1076}, the intrinsic derivative 
\[D_\rho (Y,Z)^{t,(u,v),\rho}(y,z):=\big(D_\rho Y_{s,i,j}^{t,(u,v),\rho}(y,z),D_\rho Z_{s,j}^{t,(u,v),\rho}(y,z)\big)_{s\in[t,T],1\leq i\leq d,1\leq j\leq d+1}\]
is the unique solution to the SDE
\begin{equation}\label{lifted_intrinsic_sde}
    \left\{
    \begin{aligned} 
        D_\rho Y_{s,i,j}^{t,(u,v),\rho}(y,z)=\,&\sum_{k=1}^d\int_t^s\partial_{y_k} b_i(Y_r^{t,u,\rho},\tilde{\mu}_r^{t,\mu})D_\rho Y_{r,k,j}^{t,(u,v),\rho}(y,z)\,\d r\\
            &\,+\sum_{k,\,l=1}^d\int_t^s\partial_{y_k} \sigma_{i,l}(Y_r^{t,u,\rho},\tilde{\mu}_r^{t,\mu})D_\rho Y_{r,k,j}^{t,(u,v),\rho}(y,z)\,\d W_r^{l}\\
            &\,+\sum_{k=1}^d\int_t^s\Expect\Big[(D_\mu b_i)_k(y^\prime,\tilde{\mu}_r^{t,\mu},Y_r^{t,y,\rho})Z_r^{t,(y,z),\rho}\partial_{y_j}Y_{r,k}^{t,y,\rho}\\
            &\,+(D_\mu b_i)_k(y^\prime,\tilde{\mu}_r^{t,\mu},Y_r^{t,\xi})Z_r^{t,\xi}U_{r,k,j}^{t,\xi}(y,z)\Big]\big|_{y^\prime=Y_r^{t,u,\rho}}\,\d r,\\
            &\,+\sum_{k,\,l=1}^d\int_t^s\Expect\Big[(D_\mu \sigma_{i,l})_k(y^\prime,\tilde{\mu}_r^{t,\mu},Y_r^{t,y,\rho})Z_r^{t,(y,z),\rho}\partial_{y_j}Y_{r,k}^{t,y,\rho}\\
            &\,+(D_\mu \sigma_{i,l})_k(y^\prime,\tilde{\mu}_r^{t,\mu},Y_r^{t,\xi})Z_r^{t,\xi}U_{r,k,j}^{t,\xi}(y,z)\Big]\big|_{y^\prime=Y_r^{t,u,\rho}}\,\d W_r^{l},\\
            &\,+\int_t^s\Expect\Big[\delta_\mu b_i(y^\prime,\tilde{\mu}_r^{t,\mu},Y_r^{t,y,\rho})\partial_{y_j}Z_r^{t,(y,z),\rho}\\
            &\,+\delta_\mu b_i(y^\prime,\tilde{\mu}_r^{t,\mu},Y_r^{t,\xi})U_{r,d+1,j}^{t,\xi}(y,z)\Big]\big|_{y^\prime=Y_r^{t,u,\rho}}\,\d r,\\
            &\,+\sum_{l=1}^d\int_t^s\Expect\Big[\delta_\mu \sigma_{i,l}(y^\prime,\tilde{\mu}_r^{t,\mu},Y_r^{t,y,\rho})\partial_{y_j}Z_r^{t,(y,z),\rho}\\
            &\,+(\delta_\mu \sigma_{i,l}(y^\prime,\tilde{\mu}_r^{t,\mu},Y_r^{t,\xi})U_{r,d+1,j}^{t,\xi}(y,z)\Big]\big|_{y^\prime=Y_r^{t,u,\rho}}\,\d W_r^{l},\\
            &\qquad\qquad\qquad\qquad\qquad\qquad\qquad\qquad\qquad \quad i=1,...,d,\\
        D_\rho Z_{s,j}^{t,(u,v),\rho}(y,z)=&\,\sum_{k=1}^d\int_t^s\partial_{y_k} c(Y_r^{t,u,\rho},\tilde{\mu}_r^{t,\mu})Z_r^{t,(u,v),\rho} D_\rho Y_{r,k,j}^{t,(u,v),\rho}(y,z)\,\d r\\
            &\,+\int_t^s c(Y_r^{t,u,\rho},\tilde{\mu}_r^{t,\mu}) D_\rho Z_{r,j}^{t,(u,v),\rho}(y,z)\,\d r\\
            &\,+\sum_{k=1}^d\int_t^s\Expect\Big[z^\prime(D_\mu c)_k(y^\prime,\tilde{\mu}_r^{t,\mu},Y_r^{t,y,\rho})Z_r^{t,(y,z),\rho}\partial_{y_j}Y_{r,k}^{t,y,\rho}\\
            &\,+z^\prime(D_\mu c)_k(y^\prime,\tilde{\mu}_r^{t,\mu},Y_r^{t,\xi})Z_r^{t,\xi}U_{r,k,j}^{t,\xi}(y,z)\Big]\big|_{(y^\prime,z^\prime)=(Y,Z)_r^{t,(u,v),\rho}}\,\d r,\\
            &\,+\int_t^s\Expect\Big[z^\prime\delta_\mu c(y^\prime,\tilde{\mu}_r^{t,\mu},Y_r^{t,y,\rho})\partial_{y_j}Z_r^{t,(y,z),\rho}\\
            &\,+z^\prime\delta_\mu c(y^\prime,\tilde{\mu}_r^{t,\mu},Y_r^{t,\xi})U_{r,d+1,j}^{t,\xi}(y,z)\Big]\big|_{(y^\prime,z^\prime)=(Y,Z)_r^{t,(u,v),\rho}}\,\d r,\\
    \end{aligned}
    \right.
\end{equation}
where $U^{t,\xi}(y,z)=\big(U_{s,i,j}^{t,\xi}(y,z)\big)_{s\in[t,T];\,i,j=1,...,d+1}$ is that of the SDE
\begin{equation}\label{lifting_intrinsic_auxiliary}
    \left\{
    \begin{aligned}
        U_{s,i,j}^{t,\xi}(y,z)=&\,\sum_{k=1}^d\int_t^s\partial_{y_k} b_i(Y_r^{t,\xi},\tilde{\mu}_r^{t,\mu})U_{r,k,j}^{t,\xi}(y,z)\,\d r\\
            &\,+\sum_{k,\,l=1}^d\int_t^s\partial_{y_k} \sigma_{i,l}(Y_r^{t,\xi},\tilde{\mu}_r^{t,\mu})U_{r,k,j}^{t,\xi}(y,z)\,\d W_r^{l}\\
            &\,+\sum_{k=1}^d\int_t^s\Expect\Big[(D_\mu b_i)_k(y^\prime,\tilde{\mu}_r^{t,\mu},Y_r^{t,y,\rho})Z_r^{t,(y,z),\rho}\partial_{y_j}Y_{r,k}^{t,y,\rho}\\
            &\,+(D_\mu b_i)_k(y^\prime,\tilde{\mu}_r^{t,\mu},Y_r^{t,\xi})Z_r^{t,\xi}U_{r,k,j}^{t,\xi}(y,z)\Big]\big|_{y^\prime=Y_r^{t,\xi}}\,\d r,\\
            &\,+\sum_{k,\,l=1}^d\int_t^s\Expect\Big[(D_\mu \sigma_{i,l})_k(y^\prime,\tilde{\mu}_r^{t,\mu},Y_r^{t,y,\rho})Z_r^{t,(y,z),\rho}\partial_{y_j}Y_{r,k}^{t,y,\rho}\\
            &\,+(D_\mu \sigma_{i,l})_k(y^\prime,\tilde{\mu}_r^{t,\mu},Y_r^{t,\xi})Z_r^{t,\xi}U_{r,k,j}^{t,\xi}(y,z)\Big]\big|_{y^\prime=Y_r^{t,\xi}}\,\d W_r^{l},\\
            &\,+\int_t^s\Expect\Big[\delta_\mu b_i(y^\prime,\tilde{\mu}_r^{t,\mu},Y_r^{t,y,\rho})\partial_{y_j}Z_r^{t,(y,z),\rho}\\
            &\,+\delta_\mu b_i(y^\prime,\tilde{\mu}_r^{t,\mu},Y_r^{t,\xi})U_{r,d+1,j}^{t,\xi}(y,z)\Big]\big|_{y^\prime=Y_r^{t,\xi}}\,\d r,\\
            &\,+\sum_{l=1}^d\int_t^s\Expect\Big[\delta_\mu \sigma_{i,l}(y^\prime,\tilde{\mu}_r^{t,\mu},Y_r^{t,y,\rho})\partial_{y_j}Z_r^{t,(y,z),\rho}\\
            &\,+(\delta_\mu \sigma_{i,l}(y^\prime,\tilde{\mu}_r^{t,\mu},Y_r^{t,\xi})U_{r,d+1,j}^{t,\xi}(y,z)\Big]\big|_{y^\prime=Y_r^{t,\xi}}\,\d W_r^{l},\\
            &\qquad\qquad\qquad\qquad\qquad\qquad\qquad\qquad\qquad \quad i=1,...,d,\\
        U_{s,d+1,j}^{t,\xi}(y,z)=&\,\sum_{k=1}^d\int_t^s\partial_{y_k} c(Y_r^{t,\xi},\tilde{\mu}_r^{t,\mu})Z_r^{t,\xi} U_{r,k,j}^{t,\xi}(y,z)\,\d r\\
            &\,+\int_t^s c(Y_r^{t,\xi},\tilde{\mu}_r^{t,\mu}) U_{r,d+1,j}^{t,\xi}(y,z)\,\d r\\
            &\,+\sum_{k=1}^d\int_t^s\Expect\Big[z^\prime(D_\mu c)_k(y^\prime,\tilde{\mu}_r^{t,\mu},Y_r^{t,y,\rho})Z_r^{t,(y,z),\rho}\partial_{y_j}Y_{r,k}^{t,y,\rho}\\
            &\,+z^\prime(D_\mu c)_k(y^\prime,\tilde{\mu}_r^{t,\mu},Y_r^{t,\xi})Z_r^{t,\xi}U_{r,k,j}^{t,\xi}(y,z)\Big]\big|_{(y^\prime,z^\prime)=(Y,Z)_r^{t,\xi}}\,\d r,\\
            &\,+\int_t^s\Expect\Big[z^\prime\delta_\mu c(y^\prime,\tilde{\mu}_r^{t,\mu},Y_r^{t,y,\rho})\partial_{y_j}Z_r^{t,(y,z),\rho}\\
            &\,+z^\prime\delta_\mu c(y^\prime,\tilde{\mu}_r^{t,\mu},Y_r^{t,\xi})U_{r,d+1,j}^{t,\xi}(y,z)\Big]\big|_{(y^\prime,z^\prime)=(Y,Z)_r^{t,\xi}}\,\d r.\\
    \end{aligned}
    \right.
\end{equation}

Here $j=1,...,d+1$, and recall $y_{d+1}=z$.
Let $V^{t,\xi}(y.z):=\partial_z U^{t,\xi}(y.z)$, and it holds by \eqref{lifted_solution} that
\begin{equation}\label{lifting_intrinsic_auxiliary02}
    \begin{aligned}
        V_{s,i,d+1}^{t,\xi}(y,z)=&\,\sum_{k=1}^d\int_t^s\partial_{y_k} b_i(Y_r^{t,\xi},\tilde{\mu}_r^{t,\mu})V_{r,k,d+1}^{t,\xi}(y,z)\,\d r\\
            &\,+\sum_{k,\,l=1}^d\int_t^s\partial_{y_k} \sigma_{i,l}(Y_r^{t,\xi},\tilde{\mu}_r^{t,\mu})V_{r,k,d+1}^{t,\xi}(y,z)\,\d W_r^{l}\\
            &\,+\sum_{k=1}^d\int_t^s\Expect\Big[(D_\mu b_i)_k(y^\prime,\tilde{\mu}_r^{t,\mu},Y_r^{t,\xi})Z_r^{t,\xi}V_{r,k,d+1}^{t,\xi}(y,z)\Big]\big|_{y^\prime=Y_r^{t,\xi}}\,\d r,\\
            &\,+\sum_{k,\,l=1}^d\int_t^s\Expect\Big[(D_\mu \sigma_{i,l})_k(y^\prime,\tilde{\mu}_r^{t,\mu},Y_r^{t,\xi})Z_r^{t,\xi}V_{r,k,d+1}^{t,\xi}(y,z)\Big]\big|_{y^\prime=Y_r^{t,\xi}}\,\d W_r^{l},\\
            &\,+\int_t^s\Expect\Big[\delta_\mu b_i(y^\prime,\tilde{\mu}_r^{t,\mu},Y_r^{t,\xi})V_{r,d+1,d+1}^{t,\xi}(y,z)\Big]\big|_{y^\prime=Y_r^{t,\xi}}\,\d r,\\
            &\,+\sum_{l=1}^d\int_t^s\Expect\Big[(\delta_\mu \sigma_{i,l}(y^\prime,\tilde{\mu}_r^{t,\mu},Y_r^{t,\xi})V_{r,d+1,d+1}^{t,\xi}(y,z)\Big]\big|_{y^\prime=Y_r^{t,\xi}}\,\d W_r^{l},\\
            &\qquad\qquad\qquad\qquad\qquad\qquad\qquad\qquad\qquad \quad i=1,...,d,\\
        V_{s,d+1,d+1}^{t,\xi}(y,z)=&\,\sum_{k=1}^d\int_t^s\partial_{y_k} c(Y_r^{t,\xi},\tilde{\mu}_r^{t,\mu})Z_r^{t,\xi} V_{r,k,d+1}^{t,\xi}(y,z)\,\d r\\
            &\,+\int_t^s c(Y_r^{t,\xi},\tilde{\mu}_r^{t,\mu}) V_{r,d+1,d+1}^{t,\xi}(y,z)\,\d r\\
            &\,+\sum_{k=1}^d\int_t^s\Expect\Big[z^\prime(D_\mu c)_k(y^\prime,\tilde{\mu}_r^{t,\mu},Y_r^{t,\xi})Z_r^{t,\xi}V_{r,k,d+1}^{t,\xi}(y,z)\Big]\big|_{(y^\prime,z^\prime)=(Y,Z)_r^{t,\xi}}\,\d r,\\
            &\,+\int_t^s\Expect\Big[z^\prime\delta_\mu c(y^\prime,\tilde{\mu}_r^{t,\mu},Y_r^{t,\xi})V_{r,d+1,d+1}^{t,\xi}(y,z)\Big]\big|_{(y^\prime,z^\prime)=(Y,Z)_r^{t,\xi}}\,\d r.\\
    \end{aligned}
\end{equation}
Hence, we have $(V_{s,i,d+1}^{t,\xi})_{1\leq i\leq d}\equiv0$, and thus $(U_{s,i,d+1}^{t,\xi}(y,z))_{1\leq i\leq d}=(U_{s,i,d+1}^{t,\xi}(y))_{1\leq i\leq d}$.
On the other hand, for $1\leq j\leq d$, it holds analogously that
\[
    \begin{aligned}
        &\quad V_{s,i,j}^{t,\xi}(y,z)=\sum_{k=1}^d\int_t^s\partial_{y_k} b_i(Y_r^{t,\xi},\tilde{\mu}_r^{t,\mu})V_{r,k,j}^{t,\xi}(y,z)\,\d r
            +\sum_{k,\,l=1}^d\int_t^s\partial_{y_k} \sigma_{i,l}(Y_r^{t,\xi},\tilde{\mu}_r^{t,\mu})V_{r,k,j}^{t,\xi}(y,z)\,\d W_r^{l}\\
            &\,+\sum_{k=1}^d\int_t^s\Expect\Big[(D_\mu b_i)_k(y^\prime,\tilde{\mu}_r^{t,\mu},Y_r^{t,y,\rho})F_r^{t,y,\rho}\partial_{y_j}Y_{r,k}^{t,y,\rho}
            +(D_\mu b_i)_k(y^\prime,\tilde{\mu}_r^{t,\mu},Y_r^{t,\xi})Z_r^{t,\xi}V_{r,k,j}^{t,\xi}(y,z)\Big]\big|_{y^\prime=Y_r^{t,\xi}}\,\d r,\\
            &\,+\sum_{k,\,l=1}^d\int_t^s\Expect\Big[(D_\mu \sigma_{i,l})_k(y^\prime,\tilde{\mu}_r^{t,\mu},Y_r^{t,y,\rho})F_r^{t,y,\rho}\partial_{y_j}Y_{r,k}^{t,y,\rho}\\
            &\,+(D_\mu \sigma_{i,l})_k(y^\prime,\tilde{\mu}_r^{t,\mu},Y_r^{t,\xi})Z_r^{t,\xi}V_{r,k,j}^{t,\xi}(y,z)\Big]\big|_{y^\prime=Y_r^{t,\xi}}\,\d W_r^{l},\\
            &\,+\int_t^s\Expect\Big[\delta_\mu b_i(y^\prime,\tilde{\mu}_r^{t,\mu},Y_r^{t,y,\rho})\partial_{y_j}F_r^{t,y,\rho}
            +\delta_\mu b_i(y^\prime,\tilde{\mu}_r^{t,\mu},Y_r^{t,\xi})V_{r,d+1,j}^{t,\xi}(y,z)\Big]\big|_{y^\prime=Y_r^{t,\xi}}\,\d r,\\
            &\,+\sum_{l=1}^d\int_t^s\Expect\Big[\delta_\mu \sigma_{i,l}(y^\prime,\tilde{\mu}_r^{t,\mu},Y_r^{t,y,\rho})\partial_{y_j}F_r^{t,y,\rho}
            +(\delta_\mu \sigma_{i,l}(y^\prime,\tilde{\mu}_r^{t,\mu},Y_r^{t,\xi})V_{r,d+1,j}^{t,\xi}(y,z)\Big]\big|_{y^\prime=Y_r^{t,\xi}}\,\d W_r^{l},
    \end{aligned}
\]
for $i=1,...,d$, and that
\[
    \begin{aligned}
        &\quad V_{s,d+1,j}^{t,\xi}(y,z)=\sum_{k=1}^d\int_t^s\partial_{y_k} c(Y_r^{t,\xi},\tilde{\mu}_r^{t,\mu})Z_r^{t,\xi} V_{r,k,j}^{t,\xi}(y,z)\,\d r
            +\int_t^s c(Y_r^{t,\xi},\tilde{\mu}_r^{t,\mu}) V_{r,d+1,j}^{t,\xi}(y,z)\,\d r\\
            &\,+\sum_{k=1}^d\int_t^s\Expect\Big[z^\prime(D_\mu c)_k(y^\prime,\tilde{\mu}_r^{t,\mu},Y_r^{t,y,\rho})F_r^{t,y,\rho}\partial_{y_j}Y_{r,k}^{t,y,\rho}\\
            &\,+z^\prime(D_\mu c)_k(y^\prime,\tilde{\mu}_r^{t,\mu},Y_r^{t,\xi})Z_r^{t,\xi}V_{r,k,j}^{t,\xi}(y,z)\Big]\big|_{(y^\prime,z^\prime)=(Y,Z)_r^{t,\xi}}\,\d r,\\
            &\,+\int_t^s\Expect\Big[z^\prime\delta_\mu c(y^\prime,\tilde{\mu}_r^{t,\mu},Y_r^{t,y,\rho})\partial_{y_j}F_r^{t,y,\rho}
            +z^\prime\delta_\mu c(y^\prime,\tilde{\mu}_r^{t,\mu},Y_r^{t,\xi})V_{r,d+1,j}^{t,\xi}(y,z)\Big]\big|_{(y^\prime,z^\prime)=(Y,Z)_r^{t,\xi}}\,\d r,\\
    \end{aligned}
\]
which yields $\partial_z V_{s,i,j}^{t,\xi}(y,z)=0$ and thus $V_{s,i,j}^{t,\xi}(y,z)=V_{s,i,j}^{t,\xi}(y)$ for $j=1,2,...,d$.
Note that if we let $z=0$ in \eqref{lifting_intrinsic_auxiliary}, then analogously to \eqref{lifting_intrinsic_auxiliary02} we have $U_{s,i,j}^{t,\xi}(y,0)\equiv0$ for all $1\leq i\leq d+1,1\leq j\leq d$.
Hence, combining the above argument with \eqref{lifting_intrinsic_auxiliary}, we have
\begin{equation}\label{lifting_intrinsic_auxiliary_simplified}
    \left\{
        \begin{aligned}
            U_{s,i,j}^{t,\xi}(y,z)=\,&zV_{s,i,j}^{t,\xi}(y),\quad 1\leq j\leq d,\\
            U_{s,i,d+1}^{t,\xi}(y,z)=\,&U_{s,i,d+1}^{t,\xi}(y).
        \end{aligned}
    \right.
\end{equation}

Applying \eqref{lifting_intrinsic_auxiliary_simplified} to \eqref{lifted_intrinsic_sde}, we can similarly obtain
\begin{equation}\label{lifting_intrinsic_sde_simplified}
    \left\{
        \begin{aligned}
            D_\rho Y_{s,i,j}^{t,(u,v),\rho}(y,z)=\,&zA_{s,i,j}^{t,(u,v),\rho}(y),\quad 1\leq i\leq d,\,1\leq j\leq d,\\
            D_\rho Z_{s,j}^{t,(u,v),\rho}(y,z)=\,&zB_{s,j}^{t,(u,v),\rho}(y),\quad 1\leq j\leq d,\\
            D_\rho Y_{s,i,d+1}^{t,(u,v),\rho}(y,z)=\,&D_\rho Y_{s,i,d+1}^{t,(u,v),\rho}(y),\quad 1\leq i\leq d,\\
            D_\rho Z_{s,d+1}^{t,(u,v),\rho}(y,z)=\,&D_\rho Z_{s,d+1}^{t,(u,v),\rho}(y),
        \end{aligned}
    \right.
\end{equation}
where 
\[A_{s,i,j}^{t,(u,v),\rho}(y):=\partial_z D_\rho Y_{s,i,j}^{t,(u,v),\rho}(y,z),\quad B_{s,j}^{t,(u,v),\rho}(y):=\partial_z D_\rho Z_{s,j}^{t,(u,v),\rho}(y,z),\]
for $1\leq j\leq d$.
Applying \eqref{lifted_solution} and \eqref{lifting_intrinsic_sde_simplified} to \eqref{intrinsic_derivative_lifted}, we finish the proof of Lemma~\ref{intrinsic_is_linear}. 
    
\end{appendix}

\bibliographystyle{plain}
\bibliography{myref.bib}

\end{document}